\documentclass[a4paper,12pt]{article}
\usepackage{amsmath}
\usepackage{amssymb}
\usepackage{amsthm}
\usepackage[english]{babel}

\usepackage{centernot}
\usepackage{paralist}
\usepackage{enumerate}
\setdefaultleftmargin{18pt}{}{}{}{}{}

\newtheorem{theor}{Theorem}[section]

\newtheorem{lemma}[theor]{Lemma}
\newtheorem{cor}[theor]{Corollary}
\newtheorem{conj}[theor]{Conjecture}

\newtheorem{defi}[theor]{Definition}
\newtheorem*{thm}{Theorem \ref{t15}}

\renewcommand{\labelenumi}{\bf\arabic{enumi})}
\def\theenumi{(\roman{enumi})}

\newcommand{\Ind}{\mathrm{Ind}}
\newcommand{\Irr}{\mathrm{Irr}}
\newcommand{\N}{\mathbb{N}}
\newcommand{\Z}{\mathbb{Z}}

\begin{document}

\begin{center}
{\Large 
$p$-vanishing conjugacy classes of symmetric groups}

Lucia Morotti
\end{center}

\begin{abstract}
For a prime $p$, we say that a conjugacy class of a finite group $G$ is $p$-vanishing if every irreducible character of $G$ of degree divisible by $p$ takes value 0 on that conjugacy class. In this paper we completely classify 2-vanishing and 3-vanishing conjugacy classes for the symmetric group and do some work in the classification of $p$-vanishing conjugacy classes of the symmetric group for $p\geq 5$. This answers a question by Navarro for $p=2$ and $p=3$ and partly answers it for $p\geq 5$.
\end{abstract}

\section{Introduction}

Let $p$ be a prime and $n$ a non-negative integer. The work presented here started from the following question of Navarro to Olsson (December 2010):

\noindent ``What are the elements $x$ of the symmetric group $S_n$ such that $\chi(x)=0$ for all $\chi\in\Irr(S_n)$ of degree divisible by $p$?''

We start with some definitions.

\begin{defi}[$p$-singular character]
Let $\chi$ be an irreducible character of a finite group and let $p$ be a prime. We say that $\chi$ is
\emph{$p$-singular} if $p$ divides its degree.
\end{defi}

\begin{defi}[$p$-vanishing class]
A conjugacy class of a finite group $G$ is called \emph{$p$-vanishing} if all
$p$-singular characters of $G$ take value 0 on that conjugacy
class.
\end{defi}

We will say that a partition of $n$ is $p$-vanishing if it labels a $p$-vanishing conjugacy class of $S_n$.

Let $n=a_kp^k+\ldots+a_0$ be the $p$-adic decomposition of $n$ (with
$a_k\not=0$). We will also often fix some $t\geq 0$ and write $n=d_tp^t+e_t$
with $d_t\geq 0$ and $0\leq e_t<p^t$. This notation is now fix and will be used throughout the paper. Notice that, if $t\leq k$, then $d_t=a_kp^{k-t}+\ldots+a_t$ and that
$e_t=a_{t-1}p^{t-1}+\ldots+a_0$ (while, if $t>k$, then $d_t=0$ and $e_t=n$).

\begin{defi}[Partition of $p$-adic type]
A partition of $n$ is of \emph{$p$-adic type} if it is of the form
\[\left(f_{k,1}p^k,\ldots,f_{k,h_k}p^k,\ldots,f_{0,1},\ldots,f_{0,h_0}\right)\]
with $(f_{i,1},\ldots,f_{i,h_i})\vdash a_i$ for $0\leq i\leq k$.
\end{defi}

In this definition, and throughout this paper, $(f_{i,1},\ldots,f_{i,h_i})\vdash a_i$ means that $(f_{i,1},\ldots,f_{i,h_i})$ is a partition of $a_i$. For example the partition $\lambda_{n,p}:=((p^k)^{a_k},\ldots,1^{a_0})$ is a
partition of $p$-adic type. As $0\leq a_i<p$ for $0\leq i\leq k$ we have that a
partition $\alpha$ is of $p$-adic type if and only if
\[\sum_{{j:p^i|\alpha_j,}\atop{p^{i+1}\centernot|\alpha_j}}\alpha_j=a_ip^i\]
for $0\leq i\leq k$.

In \cite{b5}, Malle, Navarro and Olsson proved that $\lambda_{n,p}$ is
$p$-vanishing. For $\alpha,\beta\vdash n$ let $\chi^\alpha$ be the irreducible character of $S_n$ labeled by $\alpha$ and let $\chi^\alpha_\beta$ be the value of $\chi^\alpha$ on the conjugacy class labeled by $\beta$. In Theorem 1.4 of \cite{m2} the following classification of
$p$-singular irreducible characters for $S_n$ was proved.

\begin{theor}\label{c5}
Let $\alpha\vdash n$. The following are equivalent:
\begin{enumerate}
\def\theenumi{(\roman{enumi})}
\renewcommand{\labelenumi}{(\roman{enumi})}
\item\label{i5}
$\chi^\alpha$ is $p$-singular.

\item\label{i3}
$\chi^\alpha_\beta=0$ for every $\beta\vdash n$ of $p$-adic type.

\item\label{i2}
$\chi^\alpha_{\lambda_{n,p}}=0$.

\item\label{i4}
We cannot remove from $\alpha$ a sequence of hooks of lengths given, in order,
by the parts of $\lambda_{n,p}$.
\end{enumerate}
\end{theor}

In this paper we will often use the equivalence of \ref{i5} and \ref{i4}. Also
the following holds (Corollary 1.5 of \cite{m2}).

\begin{cor}\label{t12}
Partitions of $p$-adic type are $p$-vanishing.
\end{cor}

Going back to Navarro's question we can now ask: do there exist $p$-vanishing
partitions which are not of $p$-adic type?

For $p=2$ and $p=3$ the answer to the above question is yes, even if
$p$-vanishing partitions are quite close to being of $p$-adic type (they can
only differ from partitions of $p$-adic type on their small parts), as can be
seen in Theorem \ref{t1}. For $p\geq 5$ the author's conjecture is that there do
not exist $p$-vanishing partitions which are not of $p$-adic type (Conjecture
\ref{c3}).

\begin{theor}\label{t1}
Assume that $p=2$ and $r=3$ or that $p=3$ and $r=2$. Also assume that $n\geq 0$. Then a partition $(c_1,\ldots,c_h)$ of $n$ is $p$-vanishing if and only if we can
find $0\leq i\leq h$ such that $(c_1,\ldots,c_i)\vdash d_rp^r$ is of $p$-adic
type and $(c_{i+1},\ldots,c_h)\vdash e_r$ is $p$-vanishing.
\end{theor}

Notice that whenever $(c_1,\ldots,c_i)\vdash d_rp^r$ is of $p$-adic type with
$c_i>0$ and $(c_{i+1},\ldots,c_h)\vdash e_r$ then $(c_1,\ldots,c_h)$ is a
partition of $n$, since $c_i\geq p^r>c_{i+1}$.

It is easy, for $n<8$ if $p=2$ or for $n<9$ if $p=3$, to find which partitions
of $n$ are $p$-vanishing, as this can be done by simply looking at the character
table of $S_n$. For completeness we write such partitions in the following
table, where partitions not of $p$-adic type are in bold.

\vspace{12pt}
\noindent
\begin{tabular}{|l|l|}
\hline
$p$&$p$-vanishing partitions\\
\hline
2&(0),(1),(2),\textbf{(1,1)},(2,1),(4),\textbf{(2,1,1)},(4,1),(4,2),\textbf{(4,1,1)},(4,2,1)\\
\hline
3&(0),(1),(2),(1,1),(3),\textbf{(2,1)},\textbf{(1,1,1)},(3,1),(3,2),(3,1,1),\textbf{(4,1)},\textbf{(2,1,1,1)},\\
&(6),(3,3),\textbf{(3,2,1)},\textbf{(3,1,1,1)},(6,1),(3,3,1),(6,2),(6,1,1),(3,3,2),(3,3,1,1),\\
&\textbf{(4,3,1)},\textbf{(3,2,1,1,1)}\\
\hline
\end{tabular}

\vspace{12pt}

For a prime $p\geq 5$ we have the following conjectures.

\begin{conj}\label{c3}
If $p\geq 5$, then a partition is $p$-vanishing if and only if it is of $p$-adic
type.
\end{conj}

\begin{conj}\label{c4}
Let $p\geq 5$. Then for every $n$ and any $p$-vanishing partition
$(c_1,\ldots,c_h)$ of $n$ we have that
\[\sum_{i:c_i<a_0}c_i\leq a_0.\]
\end{conj}

Even if the second conjecture seems weaker than the first one, they turn out to be equivalent.

\begin{theor}\label{t3}
Conjectures \ref{c3} and \ref{c4} are equivalent.
\end{theor}

\section{Some definitions and basic lemmas}

In this section we will give some results which will be used later in proving
Theorems \ref{t1} and \ref{t3}. For $\alpha$
is a partition and $r$ is a positive integer we will write $\alpha_{(r)}$ for
the $r$-core of $\alpha$, $\alpha^{(r)}$ for the $r$-quotient of $\alpha$ and
$w_r(\alpha)$ for the $r$-weight of $\alpha$. For definition and basic results about $r$-cores, $r$-quotients and $r$-weights see Section I.3 of \cite{b4}. We will need the following result
about partitions.

\begin{lemma}\label{l1}
Let $\alpha$ and $\beta$ be partitions and $r,s\geq 1$. If $\beta$ is obtained
from $\alpha$ by removing an $rs$-hook then $\beta$ can be obtained from
$\alpha$ by removing $r$ hooks of length $s$.
\end{lemma}

See Theorem 3.3 and Proposition 3.6 of \cite{b4}. In particular the following
holds.

\begin{cor}\label{c1}
Let $\alpha$ be a partition and $r,s$ be positive integers. If $\beta$ is
obtained from $\alpha$ by removing an $rs$-hook, then we have that
$w_r(\beta)=w_r(\alpha)-s$.
\end{cor}

\begin{defi}\label{d1}
Let $\alpha$ be a partition of $n$. For $i\geq 0$ define
\[b_i(\alpha):=w_{p^i}(\alpha)-pw_{p^{i+1}}(\alpha).\]
\end{defi}

By Corollary \ref{c1} we have that $b_i(\alpha)=w_{p^i}(\alpha_{(p^{i+1})})$, in
particular $b_i(\alpha)\geq 0$ and the following lemma holds.

\begin{lemma}\label{l'9}
If $b_i(\alpha)$ are as in the Definition \ref{d1}, then we have that, for $j\geq 0$,
\[w_{p^j}(\alpha)=\sum_{i\geq j}p^{i-j}b_i(\alpha).\]
\end{lemma}

See Proposition 4.5 of \cite{b5}.

\begin{lemma}\label{l14}
We have that $\chi^\alpha$ is not $p$-singular if and only if $b_i(\alpha)=a_i$
for every $i\geq 0$.
\end{lemma}

See Sections 3 and 4 of \cite{m3}. The following is an easy corollary to the
previous lemma.

\begin{cor}\label{c7}
If $w_{p^m}(\alpha)\not=d_m$ for some $m\geq 0$, then $\chi^\alpha$ is
$p$-singular.
\end{cor}

\begin{proof}
From Lemma \ref{l'9} we have that in this case
\[\sum_{i\geq m}b_i(\alpha)p^{i-m}=w_{p^m}(\alpha)\not=d_m=\sum_{i\geq
m}a_ip^{i-m},\]
in particular there exists $i$ with $b_i(\alpha)\not=a_i$ and so we can conclude
by Lemma \ref{l14}.
\end{proof}

\begin{cor}\label{l4}
Let $t\geq 0$ and assume that $n=d_tp^t+e_t$ with $d_t\geq 1$ and $0\leq
e_t<p^t$.  Let $\alpha$ be a partition of $n$ with $\alpha_1>\alpha_2\geq 1$ and
such that
\begin{eqnarray*}
h_{1,\alpha_2}^\alpha&>&d_tp^t,\\
h_{1,\alpha_2+1}^\alpha&>&(d_t-1)p^t,\\
h_{1,\alpha_2+1}^\alpha&<&d_tp^t,\\
h_{2,1}^\alpha&<&p^t.
\end{eqnarray*}
Then $\chi^\alpha$ is $p$-singular.
\end{cor}

\begin{proof}
From Corollary \ref{c7} it is enough to show that $w_{p^t}(\alpha)<d_t$.

As $h_{1,\alpha_2+1}^\alpha>(d_t-1)p^t$ and $h_{2,1}^\alpha<p^t$ we can remove
from $\alpha$ a sequence of $(d_t-1)$ hooks of length $p^t$ in a unique way
obtaining $\beta=(m,\alpha_2,\alpha_3,\ldots)$ for some $m>\alpha_2$ (we use
that $l_{1,\alpha_2+1}^\alpha=0$). As
\begin{eqnarray*}
h_{1,\alpha_2}^\beta&=&h_{1,\alpha_2}^\alpha-(d_t-1)p^t>p^t,\\
h_{1,\alpha_2+1}^\beta&=&h_{1,\alpha_2+1}^\alpha-(d_t-1)p^t<p^t,\\
h_{2,1}^\beta&=&h_{2,1}^\alpha<p^t
\end{eqnarray*}
we can not remove from $\beta$ any further hook of length $p^t$. In particular
$w_{p^t}(\alpha)=d_t-1$ and so the corollary follows.
\end{proof}

In particular the following corollary holds.

\begin{cor}\label{l16}
Let $t\geq 0$ and assume that $d_t,e_t\not=0,n$ and $\alpha=(c,1^{n-c})$, with
$e_t\leq n-c<p^t$. Then $\chi^\alpha$ is $p$-singular.
\end{cor}

\begin{proof}
As
\begin{eqnarray*}
h_{1,1}^\alpha&=&n>d_tp^t,\\
h_{1,2}^\alpha&=&c-1<n-e_t=d_tp^t,\\
h_{1,2}^\alpha&=&c-1\geq n-p^t=(d_t-1)p^t+e_t>(d_t-1)p^t,\\
h_{2,1}^\alpha&=&n-c<p^t
\end{eqnarray*}
the corollary follows from Corollary \ref{l4}.
\end{proof}

We will now give an additional equivalent condition for $\chi^\alpha$ to be $p$-singular.

\begin{defi}[Partitions of class $m$]
We say that $\alpha\vdash n$ is of \emph{class $m\geq 0$} if it isn't possible
to recursively remove from $\alpha$ a sequence of hooks with hook-lengths given
by the partition
$\left((p^k)^{a_k},(p^{k-1})^{a_{k-1}},\ldots,(p^m)^{a_m}\right)$.
\end{defi}

\begin{lemma}\label{l22}
Let $\alpha\vdash n$. Then $\chi^\alpha$ is $p$-singular if and only if $\alpha$
is of class $m$ for some $m\geq 0$.
\end{lemma}

\begin{proof}
This follows easily from Theorem \ref{c5}, as, if $\alpha$ is of class $m$ for some
$m\geq 0$, then it is also of class 0.
\end{proof}

\section{Proof of Theorems \ref{t1} and \ref{t3}}

We will now classify $p$-vanishing conjugacy classes for $p=2$ and $p=3$,
proving Theorem \ref{t1}, and for $p\geq 5$ prove Theorem \ref{t3}. Some theorems appearing in this section will be proved in
later sections, as their proofs are quite long.

The next theorem states that, if $(c_1,\ldots,c_h)$ is $p$-vanishing and $t\in\N$,
then, under certain conditions, $c_i$ is divisible by $p^t$ whenever $c_i\geq
p^t$.

\begin{theor}\label{t15}
Let $(c_1,\ldots,c_h)\vdash n$ be $p$-vanishing. If $\sum_{c_i\geq
p^t}c_i=d_tp^t$, then $c_i$ is a multiple of $p^t$ whenever $c_i\geq p^t$.
\end{theor}

The proof of this theorem can be found in Section \ref{s1}.

For $p=2$ and $p=3$ we will prove in the next two theorems that there exist some $m\in\N$ such that $\sum_{c_i\geq
p^t}c_i=d_tp^t$ for $t\geq m$ and for $(c_1,\ldots,c_h)\vdash n$ a $p$-vanishing partition. For $p\geq 5$ we will prove in the next theorem that $\sum_{c_i\geq p^t}c_i \leq d_tp^t$ for every $t\geq 0$.

\begin{theor}\label{t17}
Let $(c_1,\ldots,c_h)\vdash n$ be $p$-vanishing.
Then $\sum_{c_i\geq p^t}c_i \leq d_tp^t$ in the following cases:
\begin{itemize}
\item
$p\neq 3$,
\item
$p=3$ and $t\geq 2$.
\end{itemize}
\end{theor}

For a proof see Section \ref{s2}.

\begin{theor}\label{t19}
Let $(c_1,\ldots,c_h)\vdash n$ be $p$-vanishing. Then $\sum_{c_i\geq p^t}c_i
\geq d_tp^t$ in the following cases:
\begin{itemize}
\item
$p=2$ and $t\geq 3$,
\item
$p=3$ and $t\geq 2$.
\end{itemize}
\end{theor}

For the proof of this theorem see Section \ref{s3} (in Section \ref{s4} we will prove a theorem used in the proof of Theorem \ref{t19}).

We will now show how characters can be evaluated on certain elements of $S_n$
containing cycles of length divisible by a fixed $r\geq 1$.

\begin{defi}
Let $\beta_1,\ldots,\beta_s$ be partitions and define $m:=|\beta_1|+\ldots+|\beta_s|$. If $m=0$, then we define
\[\chi^{((0),\ldots,(0)))}_{(0)}:=1.\]
If $m\geq 1$ and $\lambda\vdash m$, then let $k\geq 1$ be a part of $\lambda$ and
$\gamma\vdash m-k$ be obtained from $\lambda$ by removing a part of length $k$. In this case we define recursively
\[\chi^{(\beta_1,\ldots,\beta_s)}_\lambda:=\sum_{l=1}^s\sum_{{(i,j)\in[\beta_l]:
}\atop{h_{i,j}^{\beta_l}=k}}(-1)^{l_{i,j}^{\beta_l}}\chi^{(\beta_1,\ldots,\beta_
{l-1},\beta_l\setminus R_{i,j}^{\beta_l},\beta_{l+1},\ldots,\beta_s)}_\gamma.\]
\end{defi}

It can be easily shown that $\chi^{(\beta_1,\ldots,\beta_s)}_\lambda$ is well defined, that is it does not depend on the order in which the parts of $\lambda$ are removed. It can also be proved by simply applying the formula for induced characters that
\[\chi^{(\beta_1,\ldots,\beta_s)}=\Ind_{S_{|\beta_1|}\times\cdots\times S_{|\beta_s|}}^{S_m}(\chi^{\beta_1}\cdots\chi^{\beta_s}).\]

In the following $\delta_r(\alpha)$ will denote the $r$-sign of $\alpha$.

\begin{lemma}\label{l3}\label{c6}
Let $\alpha$ be a partition of $n$. Let $\gamma=(\gamma_1,\ldots,\gamma_s)\vdash
w_r(\alpha)$ and $\lambda\vdash n-rw_r(\alpha)$. Also let $\pi\in
S_{rw_r(\alpha)}$ with cycle partition $(r\gamma_1,\ldots,r\gamma_s)$ and
$\rho\in S_{\{rw_r(\alpha)+1,\ldots,n\}}$ with cycle partition $\lambda$. Then
\[\chi^\alpha(\pi\rho)=\delta_r(\alpha)\chi^{\alpha_{(r)}}_\lambda\chi^{\alpha^{
(r)}}_\gamma.\]
\end{lemma}

See 4.58 of \cite{r1}.

We still need one theorem before being able to prove Theorem \ref{t1}.

\begin{theor}\label{t20}
Let $(c_1,\ldots,c_h)$ be a partition of $n$ and let $m\geq 0$. Assume that
$d_tp^t=\sum_{p^t|c_i} c_i$ for every $m\leq t\leq k$. Let $l$ be maximal such
that $c_l\geq p^m$. Then $(c_1,\ldots,c_h)$ is $p$-vanishing if and only if
$(c_{l+1},\ldots,c_h)$ is $p$-vanishing.
\end{theor}

\begin{proof}
The theorem clearly holds if $m>k$. So we can assume that $m\leq k$.

Notice that by assumption
\[0\leq\sum_{i\leq l:p^t\centernot|c_i}c_i=\sum_{i\leq l}c_i-\sum_{i:p^t|c_i} c_i\leq
n-d_tp^t<p^t.\]
By definition of $l$ it then follows that $p^t\mid c_i$ for $i\leq l$.

By assumption we have that, for $m\leq t<k$,
\begin{equation}\label{eq7}
\sum_{{i:p^t|c_i,}\atop{p^{t+1}\centernot|c_i}}c_i=d_tp^t-d_{t+1}p^{t+1}=a_tp^t.
\end{equation}
Also, again by assumption,
\begin{equation}\label{eq8}
\sum_{i:p^k|c_i}c_i=d_kp^k=a_kp^k.
\end{equation}
So by Lemma \ref{l1} and the Murnaghan-Nakayama formula, if
$\chi^\alpha_{(c_1,\ldots,c_h)}\not=0$ we can remove from $\alpha$ a sequence of
hooks with lengths $((p^k)^{a_k},\ldots,(p^m)^{a_m})$.

First assume that $(c_{l+1},\ldots,c_h)$ is $p$-vanishing. Let $\alpha\vdash n$
with $p\mid\deg(\chi^\alpha)$ and $r\leq k$ maximal such that $\alpha$ is of
class $r$. Notice that such an $r$ exists by Lemma \ref{l22}. If $r\geq m$, then
we cannot remove from $\alpha$ a sequence of hooks with lengths
$(c_1,\ldots,c_h)$ by the previous part of the proof and so in this case
$\chi^\alpha_{(c_1,\ldots,c_h)}=0$. Assume now that $\alpha$ is of class $r$ but
not of class $m$ for some $r<m$. Then, by Lemma \ref{l1} and the remark at the
beginning of the proof, if $\beta$ is obtained from $\alpha$ by removing a
sequence of hooks of lengths $(c_1,\ldots,c_l)$ we have that
$\beta=\alpha_{(p^m)}$ (as such a $\beta$ satisfies $|\beta|<p^m$ and is
obtained from $\alpha$ by removing hooks of lengths divisible by $p^m$). As
$\alpha$ is not of class $m$ we then have that $\alpha_{(p^m)}$ can be obtained
from $\alpha$ by removing a sequence of hooks of lengths
$((p^k)^{a_k},\ldots,(p^m)^{a_m})$. As $|\alpha_{(p^m)}|=a_0+\ldots+a_
{m-1}p^{m-1}$ it follows that $\alpha_{(p^m)}$ is of class $r$ (as $\alpha$ is
of class $r$ but not of class $m$) and so, in particular, we have by Lemma
\ref{l22} that $p\mid\deg(\chi^{\alpha_{(p^m)}})$. As
$(c_{l+1},\ldots,c_h)\vdash|\alpha_{(p^m)}|$ is $p$-vanishing, it then follows
from the Murnaghan-Nakayama formula and Lemma \ref{l3} that, for some $a\in\Z$,
\[\chi^\alpha_{(c_1,\ldots,c_h)}=a\chi^{\alpha_{(p^m)}}_{(c_{l+1},\ldots,c_h)}
=0.\]
So, if $(c_{l+1},\ldots,c_h)$ is $p$-vanishing, then $(c_1,\ldots,c_h)$ is also
$p$-vanishing.

Assume now that $(c_1,\ldots,c_h)$ is $p$-vanishing. Let $\beta\vdash
|(c_{l+1},\ldots,c_h)|=e_m$ with $p\mid\deg(\chi^\beta)$ and define
$\alpha:=(\beta_1+d_mp^m,\beta_2,\beta_3,\ldots)$. As $|\beta|<p^m$ and by
assumption $d_mp^m\not=0$ so that $(1,\beta_1+1)\in[\alpha]$ and then
\begin{eqnarray*}
h_{1,\beta_1+1}^\alpha&=&d_mp^m=a_mp^m+\ldots+a_kp^k,\\
l_{1,\beta_1+1}^\alpha&=&0,
\end{eqnarray*}
we can remove from $\alpha$ a sequence of hooks with lengths
$((p^k)^{a_k},\ldots,(p^m)^{a_m})$ in a unique way obtaining $\beta$. As
$p\mid\deg(\chi^\beta)$ and the $p$-adic decomposition of $|\beta|$ is
$a_{m-1}p^{m-1}+\ldots+a_0$, we have by Theorem \ref{c5} applied to both
$\alpha$ and $\beta$ that $p$ divides the degree of $\chi^\alpha$. So, again as
$|\beta|<p^m$, $h_{1,\beta_1+1}^\alpha=a_mp^m+\ldots+a_kp^k$ and
$l_{1,\beta_1+1}^\alpha=0$, by definition of $l$ and as $(c_1,\ldots,c_h)$ is
$p$-vanishing, we have by the Murnaghan-Nakayama formula that
\[0=\chi^\alpha_{(c_1,\ldots,c_h)}=\chi^\beta_{(c_{l+1},\ldots,c_h)}\]
and as this holds for each $\beta\vdash |(c_{l+1},\ldots,c_h)|$ with
$p\mid\deg(\chi^\beta)$, we have that $(c_{l+1},\ldots,c_h)$ is $p$-vanishing.
\end{proof}

We are now ready to prove Theorem \ref{t1}.

\begin{proof}[Proof of Theorem \ref{t1}]
Let notation be as in the statement of Theorem \ref{t1} and let $l$ be maximal such that
$c_l\geq p^r$. Then by Theorems \ref{t17} and \ref{t19} we have that
\[\sum_{i:c_i\geq p^t}c_i=d_tp^t\]
for $t\geq r$. In particular by Theorem \ref{t15}
\[\sum_{i:p^t|c_i}c_i=d_tp^t\]
for $t\geq r$. So we can apply Theorem \ref{t20} and we obtain that
$(c_{l+1},\ldots,c_h)$ is $p$-vanishing. As $(c_1,\ldots,c_l)\vdash d_rp^r$ we
also have that $(c_{l+1},\ldots,c_h)\vdash e_r$. Also from Equations \eqref{eq7}
and \eqref{eq8},
\[\sum_{{i:p^t|c_i,}\atop{p^{t+1}\centernot|c_i}}c_i=a_tp^t,\]
for $t\geq r$ and so $(c_1,\ldots,c_l)\vdash d_rp^r$ is of $p$-adic type.

The other direction follows easily by Theorem \ref{t20}.
\end{proof}

We will use the following lemma in the proof of Theorem \ref{t3}.

\begin{lemma}\label{l23}
Let $t\geq 0$ and assume that $(c_1,\ldots,c_h)\vdash n$ is $p$-vanishing. If
$\sum_{c_j\geq p^t}c_j<d_tp^t$ then $\sum_{c_j<e_t}c_j>e_t$.
\end{lemma}

\begin{proof}
Assume that for some $j\not=h$ we have that $n-d_tp^t\leq c_j<p^t$ and
$\sum_{i=j+1}^h c_i\leq c_j$. Let $\alpha:=(n-c_j,1^{c_j})$. Then $\chi^\alpha$
has degree divisible by $p$ by Corollary \ref{l16}. Also as $h_{2,1}^\alpha=c_j$
when removing any sequence of hooks of lengths $(c_1,\ldots,c_h)$ from $\alpha$
we need to remove all hooks of length $>c_j$ from the first row. Let $s$ be minimal
such that $c_s=c_j$. Notice that $s\leq j$. Since
\[h_{1,2}^\alpha=n-c_j-1\geq n-c_j-c_h\geq c_1+\ldots+c_{s-1}\]
and $l_{1,2}^\alpha=0$ we can remove in a unique way the first $s-1$ hooks of
the sequence and obtain the partition $((j-s)c_j+f,1^{c_j})$, where $1\leq
f=\sum_{i=j+1}^h c_i\leq c_j$. Since $1\leq f\leq c_j$ from the
Murnaghan-Nakayama formula we have that, for $l\geq 1$,
\[\chi^{((l-1)c_j+f,1^{c_j})}_{(c_j^l,c_{j+1},\ldots,c_h)}=\left\{\begin{array}{
ll}
(-1)^{c_j-1}\chi^{(f)}_{(c_{j+1},\ldots,c_h)}&\mbox{if }l=1,\\
(-1)^{c_j-1}\chi^{((l-1)c_j+f)}_{(c_j^{l-1},c_{j+1},\ldots,c_h)}+\chi^{
((l-2)c_j+f,1^{c_j})}_{(c_j^{l-1},c_{j+1},\ldots,c_h)}&\mbox{if }l>1
\end{array}\right.\]
and so by induction on $l$ we have that
$\chi^{((l-1)c_j+f,1^{c_j})}_{(c_j^l,c_{j+1},\ldots,c_h)}=(-1)^{c_j-1}l$. In
particular
$\chi^\alpha_{(c_1,\ldots,c_h)}=\chi^{((j-s)c_j+f,1^{c_j})}_{(c_j^{j-s+1},c_{j+1
},\ldots,c_h)}\not=0$ and so $(c_1,\ldots,c_h)$ is not $p$-vanishing.

If $\sum_{c_j<e_t}c_j\leq e_t$ and
\[n-\sum_{j:c_j\geq p^t}c_j=\sum_{j:c_j<p^t}c_j>e_t\]
then we have that $(c_1,\ldots,c_h)$ has at least one part of length between $e_t$
and $p^t-1$. Let $l$ be maximal such that $e_t\leq c_l\leq p^t-1$. If $l<h$, then we
can conclude by the previous part with $j=l$ that $(c_1,\ldots,c_h)$ is not
$p$-vanishing in this case.

So assume now that $l=h$. If $c_h>e_t$ then let $\beta:=(n-c_h+1,1^{c_h-1})$.
From Corollary \ref{l16} it follows that $p\mid\deg(\chi^\beta)$. Also from the
Murnaghan-Nakayama formula we easily have that
\[\chi^\beta_{(c_1,\ldots,c_h)}=\chi^{(1^{c_h})}_{(c_h)}=(-1)^{c_h-1}\not=0\]
and so $(c_1,\ldots,c_h)$ is not $p$-vanishing. If $c_h=e_t$, then we have that $h\geq
2$ and $c_{h-1}<p^t$, since $\sum_{c_j<p^t}c_j>e_t$. So we can conclude from the
first part of the proof with $j=h-1$ that $(c_1,\ldots,c_h)$ is not
$p$-vanishing in this case either and then the lemma follows.
\end{proof}

We will now prove Theorem \ref{t3}, which states the equivalence of Conjectures \ref{c3} and \ref{c4}.

\begin{proof}[Proof of Theorem \ref{t3}]
It is clear that Conjecture \ref{c3} would imply Conjecture \ref{c4}, so we only need to prove the other direction.

We already know from Corollary \ref{t12} that if a partition is of $p$-adic type
then it is $p$-vanishing. So assume now that $(c_1,\ldots,c_h)\vdash n$ is
$p$-vanishing. Let $t\geq 1$. We can write
\[n=d_tp^t+a_{t-1}p^{t-1}+e_{t-1}.\]
Assume that
\begin{equation}\label{eq1}
\sum_{i:c_i\geq p^{t-1}}c_i\geq d_{t-1}p^{t-1}=d_tp^t+a_{t-1}p^{t-1}
\end{equation}
(notice that this condition holds for $t=1$, as then $p^{t-1}=1$). We will prove
that under this assumption, if Conjecture \ref{c4} holds, then
\begin{equation}\label{eq2}
\sum_{i:c_i\geq p^t}c_i\geq d_tp^t.
\end{equation}
From Theorems \ref{t15} and \ref{t17} and Equation \eqref{eq1} we have that
there exists $l\geq 0$, such that
\begin{equation}\label{eq3}
(c_1,\ldots,c_h)=(p^{t-1}f_1,\ldots,p^{t-1}f_l,c_{l+1},\ldots,c_h)
\end{equation}
with $(f_1,\ldots,f_l)\vdash d_{t-1}$ and $c_{l+1}<p^{t-1}$. Let $\beta\vdash
d_{t-1}$ with $p\mid\deg(\chi^\beta)$ and let $\alpha$ be the partition with
\begin{eqnarray*}
\alpha_{(p^{t-1})}&=&(e_{t-1}),\\
\alpha^{(p^{t-1})}&=&(\beta,(0),\ldots,(0)).
\end{eqnarray*}
Notice that $(e_{t-1})$ is a $p^{t-1}$-core since $e_{t-1}<p^{t-1}$. As the
$p$-adic decomposition of $d_{t-1}=\lfloor n/p^{t-1}\rfloor$ is
$a_kp^{k-t+1}+\ldots+a_{t-1}$, we have by Theorem \ref{c5} and Lemma \ref{l1}
that $p\mid\deg(\chi^\alpha)$. As $(c_1,\ldots,c_h)$ is $p$-vanishing, applying
Lemmas \ref{l3} and \ref{c6} we have that
\[0=\chi^\alpha_{(c_1,\ldots,c_h)}=\pm\chi^{(e_{t-1})}_{(c_{l+1},\ldots,c_h)}\chi^{
(\beta,(0),\ldots,(0))}_{(f_1,\ldots,f_l)}=\pm\chi^\beta_{(f_1,\ldots,f_l)}\]
and so $\chi^\beta_{(f_1,\ldots,f_l)}=0$. As this holds for every $\beta\vdash
d_{t-1}$ with $p\mid\deg(\chi^\beta)$ it follows that $(f_1,\ldots,f_l)$ is
$p$-vanishing. As $d_{t-1}=d_tp+a_{t-1}$ and we are assuming that Conjecture
\ref{c4} holds we have that $\sum_{f_i<a_{t-1}}f_i\leq a_{t-1}$. From Lemma
\ref{l23} applied to $n'=d_{t-1}=d_tp+a_{t-1}$, $t'=1$ and $(f_1,\ldots,f_l)$ it
then follows that $\sum_{f_i\geq p}f_i\geq d_t$. So Equation \eqref{eq2} follows
from Equation \eqref{eq3}.

By induction and Theorem \ref{t17} we have that
\[\sum_{i:c_i\geq p^t}c_i=d_tp^t\]
for each $t\geq 0$. So, by Theorem \ref{t15} we have that, for $t\geq 0$,
\[\sum_{{i:p^t|c_i,}\atop{p^{t+1}\centernot|c_i}}c_i=\sum_{i:c_i\geq
p^t}c_i-\sum_{i:c_i\geq p^{t+1}}c_i=a_tp^t\]
and then $(c_1,\ldots,c_h)$ is of $p$-adic type.
\end{proof}

\section{Proof of Theorem \ref{t15}}\label{s1}

We restate here Theorem \ref{t15} and then prove it.

\begin{thm}
Let $t\geq 0$ and $(c_1,\ldots,c_h)$ be $p$-vanishing with $c_h>0$. If
$\sum_{c_j\geq p^t}c_j=d_tp^t$ then $c_j$ is a multiple of $p^t$ whenever
$c_j\geq p^t$.
\end{thm}

\begin{proof}
If $d_t=0$, then the theorem clearly holds, as then $n<p^t$ and so in this case all
part of $(c_1,\ldots,c_h)$ are smaller than $p^t$. So we can assume that
$d_t>0$. Also we can assume that $p^t>1$.

Let $m:=\sum_{c_j\geq p^t}c_j$ and assume that $m=d_tp^t$ and that there exists $j$
for which $c_j\geq p^t$ but $c_j$ is not a multiple of $p^t$. We will show
that in this case $(c_1,\ldots,c_h)$ is not $p$-vanishing, giving a
contradiction with the assumptions. Let $l$ be maximal such that $c_l\geq p^t$
and $p^t\nmid c_l$. Since $p^t\mid m$, $p^t\nmid c_l$ and $c_l\geq p^t$ there
must exists by definition of $l$ and $m$ some $1\leq l'<l$ with $p^t\nmid
c_{l'}$. Since $c_{l'}\geq c_l$ as $(c_1,\ldots,c_h)$ is a partition it follows
that $d_tp^t=m\geq c_{l'}+c_l\geq 2c_l$.  Also let $s\geq 1$ minimal such that
$c_s=c_l$. Write $c_l=cp^t+f$ with $1\leq f<p^t$. By definition of $l$ and since
$m=d_tp^t$ it follows that $r:=\sum_{j>l}c_j\equiv e_t\mbox{ mod }p^t$. The
proof of this theorem will be divided in the following cases:
\begin{enumerate}
\item
$e_t=0$,

\item
$e_t>0$ and $e_t\leq f<p^t$,

\item
$e_t>0$ and $1\leq f<e_t$.
\end{enumerate}
These three cases cover all possibilities, since by assumption $p^t\nmid c_l$.
We will now study in turn the above cases, by showing that in each one of them we get a
contradiction with $(c_1,\ldots,c_h)$ being $p$-vanishing.
\begin{enumerate}
\item
In this case let $\alpha:=(n-c_l,2,1^{c_l-2})$. This is a partition of $n$ since
$n\geq 2c_l$ and $c_l>p^t$. We will first show that $p\mid\deg(\chi^\alpha)$. In
order to do this we will first show that $w_{p^t}(\alpha)<d_t$ and since $e_t=0$
to prove this it is enough to prove that the $p^t$-core is not equal to $(0)$.
Notice that since $e_t=0$ and $c_l=cp^t+f$ we have that $n-c_l\equiv p^t-f\mbox{
mod }p^t$. Also $1\leq p^t-f<p^t$. From the definition of $\alpha$ it follows that
\begin{eqnarray*}
\alpha_{(p^t)}&=&\left\{\begin{array}{ll}
(p^t-f,2,1^{f-2})_{(p^t)}&\mbox{if }f\not\in\{1,p^t-1\},\\
(p^t-f,2,1^{p^t+f-2})_{(p^t)}&\mbox{if }1=f\not=p^t-1,\\
(2p^t-f,2,1^{f-2})_{(p^t)}&\mbox{if }1\not=f=p^t-1,\\
(2p^t-f,2,1^{p^t+f-2})_{(p^t)}&\mbox{if }1=f=p^t-1
\end{array}\right.\\
&=&\left\{\begin{array}{ll}
(p^t-f,2,1^{f-2})_{(p^t)}&\mbox{if }f\not\in\{1,p^t-1\},\\
(p^t-1,2,1^{p^t-1})_{(p^t)}&\mbox{if }1=f\not=p^t-1,\\
(p^t+1,2,1^{p^t-3})_{(p^t)}&\mbox{if }1\not=f=p^t-1,\\
(p^t+1,2,1^{p^t-1})_{(p^t)}&\mbox{if }1=f=p^t-1.
\end{array}\right.
\end{eqnarray*}
If $1=f=p^t-1$, then $p^t=2$ and $\alpha_{(2)}=(3,2,1)_{(2)}=(3,2,1)$. In
particular in this case $w_{p^t}(\alpha)<d_t$. Also since $p^t>1$,
\begin{eqnarray*}
h_{1,1}^{(p^t-f,2,1^{f-2})}&=&p^t-1,\\
h_{1,1}^{(p^t-1,2,1^{p^t-1})}&=&2p^t-1,\\
h_{1,2}^{(p^t-1,2,1^{p^t-1})}&=&p^t-1,\\
h_{2,1}^{(p^t-1,2,1^{p^t-1})}&=&p^t+1,\\
h_{2,2}^{(p^t-1,2,1^{p^t-1})}&=&1,\\
h_{3,1}^{(p^t-1,2,1^{p^t-1})}&=&p^t-1
\end{eqnarray*}
and $(p^t-1,2,1^{p^t-1})$ and $(p^t+1,2,1^{p^t-3})$ are conjugate to each other,
it follows that $w_{p^t}(\alpha)<d_t$ also in the other cases. It then follows
from Corollary \ref{c7} that $p\mid\deg(\chi^\alpha)$.

We will now show that $\chi^\alpha_{(c_1,\ldots,c_h)}\not=0$, which will give a
contradiction with the assumption of $(c_1,\ldots,c_h)$ being $p$-vanishing.
Notice that in this case $p^t\mid r$. Assume first that $l<h$. Then $r\geq
p^t\geq 2$ and
\begin{eqnarray*}
h_{1,2}^{(r,2,1^{c_l-2})}&=&r\not=c_l,\\
h_{1,3}^{(r,2,1^{c_l-2})}&=&r-2
\end{eqnarray*}
(the last equation holding only if $r\geq 3$). So, from the Murnaghan-Nakayama formula,
\begin{eqnarray*}
\chi^\alpha_{(c_1,\ldots,c_h)}&=&\chi^{(r+(l-s)c_l,2,1^{c_l-2})}_{(c_s,\ldots,
c_h)}\\
&=&(l-s)(-1)^{c_l-2}\chi^{(r+c_l)}_{(c_l,\ldots,c_h)}+\chi^{(r,2,1^{c_l-2})}_{
(c_l,\ldots,c_h)}\\
&=&(l-s+1)(-1)^{c_l-2}\chi^{(r)}_{(c_{l+1},\ldots,c_h)}+\delta_{r\geq
c_l+2}\chi^{(r-c_l,2,1^{c_l-2})}_{(c_{l+1},\ldots,c_h)}\\
&=&(l-s+1)(-1)^{c_l-2}+\delta_{r\geq
c_l+2}\chi^{(r-c_l,2,1^{c_l-2})}_{(c_{l+1},\ldots,c_h)}
\end{eqnarray*}
where $\delta_{x\geq y}=1$ if $x\geq y$ and $\delta_{x\geq y}=0$ otherwise.
Since $r\equiv n\mbox{ mod }p^t$ it follows that, if $r\geq c_l+2$, then
$(r-c_l,2,1^{c_l-2})_{(p^t)}=\alpha_{(p^t)}\not=(0)$ and so, since $p^t\mid r$
we have that $w_{p^t}((r-c_l,2,1^{c_l-2}))<r/p^t$. Also since in this case $m=n$
we have that $c_j\geq p^t$ for $j\leq h$ and then by maximality of $l$ that
$p^t\mid c_j$ for $l<j\leq h$. From repeated application of Corollary \ref{c1}
it then follows that $\chi^{(r-c_l,2,1^{c_l-2})}_{(c_{l+1},\ldots,c_h)}=0$. In
particular, since $s\leq l$, if $l<h$ we have that
$\chi^\alpha_{(c_1,\ldots,c_h)}\not=0$.

Assume now that $l=h$. First assume that $s=l$. Then since $c_j>c_l$ for $1\leq
j<l$ we have from the Murnaghan-Nakayama formula that
\[\chi^\alpha_{(c_1,\ldots,c_h)}=\chi^{(c_{l-1},2,1^{c_l-2})}_{(c_{l-1},c_l)}
=-\chi^{(1^{c_l})}_{(c_l)}=(-1)^{c_l}\not=0.\]
If $s<l$, since $c_j=c_l$ for $s\leq j\leq l$ and since $c_l>1$, then
\begin{eqnarray*}
\chi^\alpha_{(c_1,\ldots,c_h)}&=&\chi^{((s-l)c_l,2,1^{c_l-2})}_{(c_l^{l-s+1})}\\
&=&(l-s-1)(-1)^{c_l-2}\chi^{(2c_l)}_{(c_l^2)}+\chi^{(c_l,2,1^{c_l-2})}_{(c_l^2)}
\\
&=&(l-s-1)(-1)^{c_l}-\chi^{(1^{c_l})}_{(c_l)}+(-1)^{c_l-2}\chi^{(c_l)}_{(c_l)}\\
&=&(l-s+1)(-1)^{c_l}\not=0.
\end{eqnarray*}

As $p\mid\deg(\chi^\alpha)$ and $\chi^\alpha_{(c_1,\ldots,c_h)}\not=0$ we have a
contradiction.

\item
In this case let $\alpha:=(n-c_l,1^{c_l})$. Notice that $\alpha$ is a partition of $n$
and that $(1,2)\in[\alpha]$ since $n\geq m\geq 2c_l$ and $c_l>p^t$. We will
first show that $p\mid\deg(\chi^\alpha)$. Let $g\equiv n-c_l-1\mbox{ mod }p^t$
with $0\leq g<p^t$. By definition $f\equiv c_l\mbox{ mod }p^t$ and $1\leq
f<p^t$. So $\alpha_{(p^t)}=(g+1,1^f)_{(p^t)}$. Since
\[h_{1,1}^{(g+1,1^f)}=f+g+1\equiv n\equiv e_t\not\equiv 0\mbox{ mod }p^t\]
and $0\leq f,g<p^t$ it follows that $\alpha_{(p^t)}=(g+1,1^f)$ and  by
assumption that
\[0\leq g=hp^t+e_t-1-f<hp^t.\]
So $h\geq 1$, that is $f+g+1>e_t$. In particular $w_{p^t}(\alpha)<d_t$ and so it
follows from Corollary \ref{c7} that $p\mid\deg(\chi^\alpha)$.

We will now show that $\chi^\alpha_{(c_1,\ldots,c_h)}\not=0$. Since $e_t>0$ and
$m=d_tp^t=n-e_t$, we have that $1\leq c_h<p^t$ and so $l<h$. Then $r>0$ and so
from the Murnaghan-Nakayama formula it follows that
\begin{eqnarray*}
\chi^\alpha_{(c_1,\ldots,c_h)}&=&\chi^{(r+(l-s)c_l,1^{c_l})}_{(c_s,\ldots,c_h)}
\\
&=&(l-s)(-1)^{c_l-1}\chi^{(r+c_l)}_{(c_l,\ldots,c_h)}+\chi^{(r,1^{c_l})}_{(c_l,
\ldots,c_h)}\\
&=&(l-s+1)(-1)^{c_l-1}\chi^{(r)}_{(c_{l+1},\ldots,c_h)}+\delta_{r>c_l}\chi^{
(r-c_l,1^{c_l})}_{(c_{l+1},\ldots,c_h)}\\
&=&(l-s+1)(-1)^{c_l-1}+\delta_{r\geq
c_l+1}\chi^{(r-c_l,1^{c_l})}_{(c_{l+1},\ldots,c_h)}.
\end{eqnarray*}
Since $l\geq s$ in order to prove that $\chi^\alpha_{(c_1,\ldots,c_h)}\not=0$ it
is enough to prove that, if $r>c_l$, then
$\chi^{(r-c_l,1^{c_l})}_{(c_{l+1},\ldots,c_h)}=0$. So assume that $r>c_l$. In
particular $r>p^t$. Since $m=d_tp^t$ and by definition of $l$, if we let $v$
be maximal such that $c_v\geq p^t$, then $l<v<h$ and $p^t\mid c_j$ for $j\leq v$.
Also we can write $r=wp^t+e_t$ with $w\geq 1$ and
\[\sum_{j>l:p^t|c_j}c_j=\sum_{j=l+1}^vc_j=wp^t,\]
as $\sum_{j>v}c_j=\sum_{c_j<p^t}c_j=n-m=e_t$. It is easy to see that
\[(r-c_l,1^{c_l})_{(p^t)}=\alpha_{(p^t)}=(g+1,1^f)\]
since $r-c_l\geq 1$ and $r\equiv n\mbox{ mod }p^t$. Since $f+g+1>e_t$ it follows
that $w_{p^t}((r-c_l,1^{c_l}))<w$ and then, from Corollary \ref{c1}, it follows
that $\chi^{(r-c_l,1^{c_l})}_{(c_{l+1},\ldots,c_h)}=0$ and so we have a
contradiction.

\item
Notice first that in this case $e_t\geq 2$, since $1\leq f<e_t$. In this case
let $\alpha:=(n-c_l,e_t,1^{c_l-e_t})$. Since $c_l\geq p^t>e_t\geq 2$ and
\[n-c_l\geq n-\sum_{j:c_j\geq p^t}c_j=n-d_tp^t=e_t\]
we have that $\alpha$ is a partition of $n$. We will show that $p$ divides the degree
of $\chi^\alpha$ and that $\chi^\alpha_{(c_1,\ldots,c_h)}\not=0$.

We will first prove that $|\alpha_{(p^t)}|\geq p^t$. It will then follow that
$w_{p^t}(\alpha)<d_t$ and so, from Corollary \ref{c7}, that the degree of
$\chi^\alpha$ is divisible by $p$. Write $n-c_l-e_t\equiv g\mbox{ mod }p$ and
$c_l-e_t\equiv h\mbox{ mod }p$ with $0\leq g,h<p^t$. It is clear that
\[\alpha_{(p^t)}=(g+e_t,e_t,1^h)_{(p^t)}.\]
Since $c_l\equiv f\mbox{ mod }p$ and $1\leq f<e_t<p^t$ we have that
\[g\equiv n-c_l-e_t\equiv -c_l\equiv -f\equiv p^t-f\mbox{ mod }p\]
and
\[h\equiv c_l-e_t\equiv f-e_t\equiv p^t+f-e_t\mbox{ mod }p.\]
As $0\leq p^t-f,p^t+f-e_t<p^t$ it then follows by assumption of $f$ and $e_t$
that $g=p^t-f$ and $h=p^t+f-e_t$. So
\[\alpha_{(p^t)}=(p^t-f+e_t,e_t,1^{p^t+f-e_t})_{(p^t)}.\]
Let $\beta:=(p^t-f+e_t,e_t,1^{p^t+f-e_t})$. As $1\leq f<e_t<p^t$ we have that
$(2,2),(3,1)\in[\beta]$. Also
\begin{eqnarray*}
h_{1,1}^\beta&=&2p^t+1,\\
h_{1,2}^\beta&=&p^t-f+e_t<2p^t,\\
h_{2,1}^\beta&=&p^t+f,\\
h_{2,2}^\beta&=&e_t<p^t,\\
h_{3,1}^\beta&=&p^t+f-e_t<p^t.
\end{eqnarray*}
So
\[A:=\{(i,j)\in[\beta]:p^t\mid h_{i,j}^\beta\}=\{(1,j)\in[\beta]:h_{1,j}=p^t\}\]
and then $w_{p^t}(\beta)=|A|\leq 1$. In particular
\[|\alpha_{(p^t)}|=|\beta_{(p^t)}|\geq 2p^t+e_t-p^t\geq p^t\]
and then $p\mid \deg(\chi^\alpha)$.

We will now prove that $\chi^\alpha_{(c_1,\ldots,c_h)}\not=0$. First notice that
by the first part of the proof of the theorem,
\[n-c_l-e_t=d_tp^t-c_l\geq c_l>0,\]
in particular $(1,e_t+1)\in[\alpha]$. As
\begin{eqnarray*}
h_{1,e_t+1}^\alpha&=&n-c_l-e_t=\sum_{j:c_j\geq
p^t}c_j-c_l\geq\sum_{j:c_j>c_l}c_j,\\
h_{2,1}^\alpha&=&c_l
\end{eqnarray*}
and as
\[n-c_l-\sum_{j:c_j>c_l}c_j=n-c_l-c_1-\ldots-c_{s-1}=(l-s)c_l+c_{l+1}
+\ldots+c_h\]
by definition of $s$, we have from the Murnaghan-Nakayama formula that
\[\chi^\alpha_{(c_1,\ldots,c_h)}=\chi^\gamma_{(c_l^{l-s+1},c_{l+1},\ldots,c_h)},
\]
where $\gamma:=((l-s)c_l+c_{l+1}+\ldots+c_h,e_t,1^{c_l-e_t})$. Since
\[c_{l+1}+\ldots+c_h\geq \sum_{j:c_j<p^t}c_j=n-\sum_{j:c_j\geq
p^t}c_j=n-d_tp^t=e_t\]
when removing from $\gamma$ a sequence of $(l-s)$ hooks of length $c_l$ we can
either remove all of them from the first row, or remove once the hook
corresponding to the node $(2,1)$ and all other hooks from the first row. So
\begin{eqnarray*}
\chi^\gamma_{(c_l^{l-s+1},c_{l+1},\ldots,c_h)}&=&\chi^{(c_{l+1}+\ldots+c_h,e_t,
1^{c_l-e_t})}_{(c_l,\ldots,c_h)}+(-1)^{c_l-e_t}(l-s)\chi^{(c_l+\ldots+c_h)}_{
(c_l,\ldots,c_h)}\\
&=&\chi^{(c_{l+1}+\ldots+c_h,e_t,1^{c_l-e_t})}_{(c_l,\ldots,c_h)}+(-1)^{c_l-e_t}
(l-s).
\end{eqnarray*}
Let $\delta:=(c_{l+1}+\ldots+c_h,e_t,1^{c_l-e_t})$ and let $v$ be maximal with
$c_v\geq p^t$. Since
\[c_{l+1}+\ldots+c_h=e_t+\sum_{j=l+1}^vc_j,\]
when removing from $\delta$ a sequence of hooks of lengths
$(c_{l+1},\ldots,c_v)$ we can either remove all of them from the first row,
obtaining $\epsilon:=(e_t^2,1^{c_l-e_t})$, which is always possible in a unique
way, or we can remove some of them from the rows below the first one and all the
other from the first one. As $h_{2,2}^\delta=e_t-1<p^t$ and as $p^t\mid c_j$ for
$l+1\leq j\leq v$, if we remove some of such hooks from rows below the first one,
then such hooks are removed from the first column and we obtain the partition
$\lambda:=(e_t+wp^t,e_t,1^{c_l-e_t-wp^t})$ for some $w\geq 1$ with
$c_l-e_t-wp^t\geq 0$. By assumptions on $c_l$ and $e_t$ we have that $w<c$.
We have
\begin{eqnarray*}
h_{1,1}^\lambda&=&c_l+1,\\
h_{1,2}^\lambda&=&e_t+wp^t<(w+1)p^t\leq cp^t<c_l,\\
h_{2,1}^\lambda&=&c_l-wp^t,
\end{eqnarray*}
in particular $\lambda$ does not have any $c_l$ hook. Since $e_t<c_l$, using the
Murnaghan-Nakayama formula it then follows that
\begin{eqnarray*}
\chi^\alpha_{(c_1,\ldots,c_h)}&=&\chi^\gamma_{(c_l^{l-s+1},c_{l+1},\ldots,c_h)}
\\
&=&\chi^\delta_{(c_l,\ldots,c_h)}+(-1)^{c_l-e_t}(l-s)\\
&=&\chi^\epsilon_{(c_l,c_{v+1},\ldots,c_h)}+(-1)^{c_l-e_t}(l-s)\\
&=&(-1)^{c_l-e_t}\chi^{(e_t)}_{(c_{v+1},\ldots,c_h)}+(-1)^{c_l-e_t}(l-s)\\
&=&(-1)^{c_l-e_t}(l-s+1)\\
&\not=&0.
\end{eqnarray*}
This gives a contradiction to $(c_1,\ldots,c_h)$ being $p$-vanishing by
assumption.
\end{enumerate}
\end{proof}

\section{Proof of Theorem \ref{t17}}\label{s2}

In order to prove Theorem \ref{t17} we will prove the following stronger version
of it.

\begin{theor}
Let $t\geq 0$ and let $(c_1,\ldots,c_h)$ be a partition of $n$. If
$(c_1,\ldots,c_h)$ is $p$-vanishing then $\sum_{c_j\geq p^t}c_j\leq d_tp^t$
unless $p=3$, $t=1$ and $n\equiv 2\mbox{ mod }3$.
\end{theor}

\begin{proof}
If $d_t=0$, then $n<p^t$ and so the theorem clearly holds. So we will now assume
that $d_t>0$. The proof of the theorem will be divided in the following cases.
In most of the cases we will assume that $\sum_{c_j\geq p^t}c_j>d_tp^t$ and
prove that then $(c_1,\ldots,c_h)$ is not $p$-vanishing. For the rest of the
proof let $m:=\sum_{c_j\geq p^t}c_j$.
\begin{enumerate}
\item
$m\leq d_tp^t$.

\item
$e_t\not=0$ and $m=n$.

\item
$e_t\not\in\{0,p^t-1\}$ and $d_tp^t<m<n$.

\item
$e_t=p^t-1\not=0$, $d_tp^t<m<n$ and $n-m>m-d_tp^t$.

\item
$e_t=p^t-1\not=0$, $d_tp^t<m<n$ and $n-m<m-d_tp^t$.

\item
$3\leq e_t=p^t-1$, $d_tp^t<m<n$ and $n-m=m-d_tp^t$.

\item
$2\geq e_t=p^t-1\not=0$, $d_tp^t<m<n$ and $n-m=m-d_tp^t$.
\end{enumerate}

These cases cover all possibilities, since, if $e_t=0$, then $m\leq n=d_tp^t$. We
will now study each case in turn.

\begin{enumerate}
\item
In this case the theorem clearly holds.

\item
Let $\alpha:=(d_tp^t,1^{e_t})$. Then $p\mid\deg(\chi^\alpha)$ by Corollary
\ref{l16}. As $h_{2,1}^\alpha=e_t<p^t$ and $c_h\geq p^t$ as $m=n$, we have that
\[\chi^\alpha_{(c_1,\ldots,c_h)}=\chi^{(c_h-e_t,1^{e_t})}_{(c_h)}=(-1)^{e_t}
\not=0\]
and so $(c_1,\ldots,c_h)$ is not $p$-vanishing in this case.

\item
In this case let $\alpha:=(d_tp^t-1,n-m+1,1^{m-d_tp^t})$. 
As $d_tp^t<m<n$, as $d_t\geq 1$ and as $e_t\leq p^t-2$ by assumption we have
that
\[2\leq \alpha_2=n-m+1\leq n-d_tp^t=e^t\leq p^t-2\leq(d_tp^t-1)-1=\alpha_1-1.\]
It then follows that $\alpha$ is a partition of $n$. Since
\[h_{2,1}^\alpha=n-d_tp^t+1=e_t+1<p^t\]
and
\[d_tp^t<m=h_{1,1}^\alpha<(d_t+1)p^t,\]
if $(i,j)$ is a node of $\alpha$ and $p^t\mid h_{i,j}^\alpha$, then $i=1$ and
$j\geq 2$. As $h_{1,2}^\alpha=d_tp^t-1$ there exist at most $d_t-1$ such nodes
$(i,j)$ and so $w_{p^t}(\alpha)<d_t$. From Corollary \ref{c7} it follows that
$p\mid\deg(\chi^\alpha)$.

We have already proved that $\alpha_2<\alpha_1$. So $(1,\alpha_2+1)=(1,n-m+2)$
is a node of $\alpha$. Let $l$ be maximal such that $c_l\geq p^t$. Then
\[\sum_{j=l+1}^hc_j=n-\sum_{j=1}^lc_j=n-m=\alpha_2-1\]
and
\begin{eqnarray*}
h_{1,n-m+2}^\alpha&=&d_tp^t-1-(n-m+2)+1\\
&=&n-(n-m)-(n-d_tp^t)-2\\
&\geq &n-\sum_{j=l+1}^hc_j-e_t-2\\
&\geq&\sum_{j\leq l}c_j-p^t\\
&\geq&\sum_{j<l}c_j.
\end{eqnarray*}
As $h_{2,1}^\alpha<p^t$ and
\[h_{(1,1)}^{(c_l-m+d_tp^t-1,n-m+1,1^{m-d_tp^t})}=c_l,\]
it follows that
\[\chi^\alpha_{(c_1,\ldots,c_h)}=\chi^{(c_l-m+d_tp^t-1,n-m+1,1^{m-d_tp^t})}_{
(c_l,\ldots,c_h)}=(-1)^{m-d_tp^t+1}\chi^{(n-m)}_{(c_{l+1},\ldots,c_h)}=\pm 1\]
and so $(c_1,\ldots,c_h)$ is not $p$-vanishing in this case.

\item
Let now $\alpha:=(d_tp^t+1,2^{m-d_tp^t},1^{n-2m+d_tp^t-1})$. Since by assumption
$d_t>0$ and $n-m>m-d_tp^t>0$ we have that $\alpha$ is a partition of $n$. Also
since $d_tp^t<m<n$ we have that $e_t\geq 2$ and so $p^t\geq 3$. In particular
$\alpha_1\geq 4>2=\alpha_2$. As
\begin{eqnarray*}
h_{1,2}^\alpha&=&m>d_tp^t,\\
h_{1,3}^\alpha&=&d_tp^t-1,\\
h_{2,1}^\alpha&=&n-m<n-d_tp^t=e_t<p^t
\end{eqnarray*}
it follows from Corollary \ref{l4} that $p\mid\deg(\chi^\alpha)$.

Let $l$ be maximal such that $c_l\geq p^t$. From $h_{2,1}^\alpha<p^t$ it also
follows that whenever we recursively remove from $\alpha$ hooks of lengths
$(c_1,\ldots,c_l)$ then all removed hooks correspond to a node on the first row
of the respective partition. Also
\[\sum_{j=l+1}^hc_j=\sum_{j:c_j<p^t}c_j=n-m=\alpha_1'\]
and
\[\sum_{j=l}^hc_j=c_l+\alpha_1'\geq p^t+\alpha_1'=n-d_tp^t+1+\alpha_1'\geq
m-d_tp^t+1+\alpha_1'=\alpha_1'+\alpha_2'.\]
So it easily follows from the Murnaghan-Nakayama formula that
\[\chi^\alpha_{(c_1,\ldots,c_h)}=\chi^{(c_l+1-m+d_tp^t,2^{m-d_tp^t},1^{
n-2m+d_tp^t-1})}_{(c_l,\ldots,c_h)}=(-1)^{m-d_tp^t}\chi^{(1^{\alpha_1'})}_{(c_{
l+1},\ldots,c_h)}=\pm1\]
and then $(c_1,\ldots,c_h)$ is not $p$-vanishing.

\item
In this case let $\alpha:=(d_tp^t+1,2^{n-m},1^{2m-n-d_tp^t-1})$. Since by
assumption $0<n-m<m-d_tp^t$ and $d_t\geq 1$ we have that $\alpha$ is a partition
of $n$. Further as $p^t=e_t+1\geq 2$ we have that $\alpha_1\geq 3>2=\alpha_2$.
As
\begin{eqnarray*}
h_{1,2}^\alpha&=&d_tp^t+n-m>d_tp^t,\\
h_{1,3}^\alpha&=&d_tp^t-1,\\
h_{2,1}^\alpha&=&m-d_tp^t<n-d_tp^t=e_t<p^t
\end{eqnarray*}
we have that $p$ divides the degree of $\chi^\alpha$ from Corollary \ref{l4}.

Again let $l$ be maximal such that $c_l\geq p^t$. Then
\[\sum_{j=l+1}^hc_j=\sum_{j:c_j<p^t}c_j=n-m=\alpha_2'-1\]
and
\[\sum_{j=l}^hc_j=c_l+\alpha_1'\geq p^t+\alpha_2'-1=n-d_tp^t+\alpha_2'\geq
m-d_tp^t+\alpha_1'=\alpha_1'+\alpha_2'\]
and so
\[\chi^\alpha_{(c_1,\ldots,c_h)}=\chi^{(c_l-\alpha_1'+1,2^{n-m},1^{2m-n-d_tp^t-1
})}_{(c_l,\ldots,c_h)}=(-1)^{\alpha_1'-1}\chi^{(1^{n-m})}_{(c_{l+1},\ldots,c_h)}
=\pm1.\]
In particular also in this case $(c_1,\ldots,c_h)$ is not $p$-vanishing.

\item
Notice that
\[e_t=n-d_tp^t=2(n-m)\]
is even and so by assumption $e_t\geq 4$. As $p^t=e_t+1$ we also have that
$t\geq 1$ and $p$ is odd. Let $\alpha:=(d_tp^t,2^{e_t/2})$. Since $e_t=p^t-1\geq
4$ we have that $(1,3)\in[\alpha]$ and
\begin{eqnarray*}
h_{1,2}^\alpha&=&d_tp^t+e_t/2-1>d_tp^t,\\
h_{1,3}^\alpha&=&d_tp^t-2,\\
h_{2,1}^\alpha&=&e_t/2+1<p^t
\end{eqnarray*}
and so $p\mid \deg(\chi^\alpha)$ by Corollary \ref{l4}. Let $l$ be maximal such
that $c_l\geq p^t$. Then
\[\sum_{j=l+1}^hc_j=\sum_{j:c_j<p^t}c_j=n-m=e_t/2\]
and
\[\sum_{j=l}^hc_j=c_l+e_t/2\geq p^t+2>e_t+2=\alpha_1'+\alpha_2'.\]
So it follows easily from the Murnaghan-Nakayama formula that
\[\chi^\alpha_{(c_1,\ldots,c_h)}=\chi^{(c_l-e_t/2,2^{e_t/2})}_{(c_l,\ldots,c_h)}
=(-1)^{e_t/2}\chi^{(1^{e_t/2})}_{(c_{l+1},\ldots,c_h)}=\pm 1.\]
In particular $(c_1,\ldots,c_h)$ is not $p$-vanishing.

\item
Also in this case we have that $e_t$ is even, since $e_t=n-d_tp^t=2(n-m)$. So by
assumption $p^t-1=e_t=2$ and then $p=3$, $t=1$ and
\[n\equiv e_t\equiv 2\mbox{ mod }3,\]
which is the exceptional case in the theorem.
\end{enumerate}
\end{proof}

\section{An additional theorem}\label{s4}

The theorem proved in this section is needed in order to prove Theorem
\ref{t19}. Since its proof is quite long we write it in a separate section.

\begin{lemma}\label{l5}
Assume that $p^t\mid n$ and that $\alpha$ is a partition of $n$. If $\alpha$ has
a $\beta$-set $\{x_1,\ldots,x_m\}$ with $m\geq 1$ and $p^t\nmid x_i$ for $1\leq
i\leq m$, then $p\mid \deg(\chi^\alpha)$.
\end{lemma}

\begin{proof}
From Corollary \ref{c7} and since $n=d_tp^t$ in this case it is enough to prove
that $\alpha_{(p^t)}\not=(0)$. Since $\{x_1,\ldots,x_m\}$ is a $\beta$-set for
$\alpha$, there exist $l_i$ for $1\leq i\leq m$ such that
$X:=\{x_1-l_1p^t,\ldots,x_m-l_mp^t\}$ is a $\beta$-set for $\alpha_{(p^t)}$. By
assumption no element of $X$ is divisible by $p^t$, in particular $0\not\in X$.
Since $|X|=m\geq 1$ we have that $\alpha_{(p^t)}\not=(0)$.
\end{proof}

The next theorem states that almost always the smallest part of a $p$-vanishing
partition is at least as large as the largest power of $p$ dividing $n$.

\begin{theor}\label{t13}
Let $(c_1,\ldots,c_h)\vdash n$, with $c_h\geq 1$, be $p$-vanishing and let
$p^t\mid n$. If $p^t\not\in\{2,3,4\}$, then $c_h\geq p^t$.

If $p^t\in\{2,3\}$ and $c_h<p^t$, then $c_h=1$.

If $p^t=4$ and $c_h<p^t$, then $(c_1,\ldots,c_h)$ is either $(2,1,1)$ or it ends
by $(f,2,1,1)$ with $f\geq 4$.
\end{theor}

\begin{proof}
The theorem is trivial if $p^t\leq 2$. So assume that $p^t\geq 3$. Also for
$n\leq 3$ the theorem is easy to prove by looking at the corresponding character
table, so we will now assume that $n\geq 4$.

Assume now that $(c_1,\ldots,c_h)$ is a partition of $n$ with $1\leq c_h<p^t$
and such that $(c_1,\ldots,c_h)$ is not in one of the special cases for
$p^t\in\{3,4\}$. We will prove that then $(c_1,\ldots,c_h)$ is not
$p$-vanishing. The proof will be divided in the following cases:
\begin{enumerate}
\item $2\leq c_h<p^t$.

\item $p^t\geq 4$ and $(c_1,\ldots,c_h)$ ends by $(1,1,1,1)$, $(2,1,1,1)$,
$(2,2,1)$, $(2,2,1,1)$, $(g,1)$ or $(g,1,1)$ with $g\geq 3$.

\item $p^t\geq 4$, $(c_1,\ldots,c_h)$ ends by $(3,2,1,1)$ or $(g,1,1,1)$ with
$g\geq 3$ but $(c_1,\ldots,c_h)\not=(3,3,1,1,1)$ and it doesn't end by
$(l,3,3,1,1,1)$.

\item $p^t\geq 4$ and $(c_1,\ldots,c_h)$ ends by $(g,2,1)$ with $g\geq 3$.

\item $p^t\geq 5$, $(c_1,\ldots,c_h)=(3,3,1,1,1)$ or it ends by $(g,3,3,1,1,1)$
or $(g,2,1,1)$ with $g\geq 4$, but $(c_1,\ldots,c_h)\not=(4,2,1,1)$ and it
doesn't end by $(l,4,2,1,1)$ with $l\geq 5$.

\item $p^t\geq 5$ and $(c_1,\ldots,c_h)=(4,2,1,1)$ or ends by $(g,4,2,1,1)$ with
$g\geq 5$.

\item $p^t=4$ and $(c_1,\ldots,c_h)=(3,3,1,1,1)$ or it ends by $(g,3,3,1,1,1)$
with $g\geq 4$.
\end{enumerate}
It can be easily checked that these cases cover all possibilities where
$c_h<p^t$ and which are not between the special cases listed for $p^t\in\{3,4\}$. We will now
study each case in turn.

\begin{enumerate}
\item
Let $\alpha:=(n-c_h,c_h)$. From $c_h<p^t$ and $p^t\mid n$ we have that $h\geq
2$. So $n\geq c_{h-1}+c_h\geq 2c_h$ and then $\alpha$ is a partition.

From $2\leq c_h<p^t$ it follows that $p^t$ divides neither $c_h$ nor
$n-c_h+1$. As $\{c_h,n-c_h+1\}$ is a $\beta$-set for $\alpha$ we have from Lemma
\ref{l5} that $p\mid\deg(\chi^\alpha)$.

We will now show that $\chi^\alpha_{(c_1,\ldots,c_h)}\not=0$. Let $s$ be minimal
with $c_s=c_h$. First assume that $s=h$. Then
\[\alpha_1=n-c_h\geq c_{h-1}>c_h=\alpha_2\]
and, since
\begin{eqnarray*}
h_{1,c_h+1}^\alpha&=&n-2c_h=c_1+\ldots+c_{h-1}-c_h>c_1+\ldots+c_{h-2},\\
h_{2,1}^\alpha&=&c_h<c_{h-1},
\end{eqnarray*}
we have that
\[\chi^\alpha_{(c_1,\ldots,c_h)}=\chi^{(c_{h-1},c_h)}_{(c_{h-1},c_h)}=-\chi^{
(c_h-1,1)}_{(c_h)}=1.\]

Assume now that $s<h$. In this case
\[h_{1,c_h+1}^\alpha=n-2c_h=c_1+\ldots+c_{h-1}-c_h\geq c_1+\ldots+c_{s-1}\]
if $(1,c_h+1)\in[\alpha]$, that is $h\geq 3$ (if $h=2$ then $s=1$, so that the
following also holds). So
\begin{eqnarray*}
\chi^\alpha_{(c_1,\ldots,c_h)}&=&\chi^{((h-s)c_h,c_h)}_{(c_h^{h-s+1})}\\
&=&(h-s-1)\chi^{(2c_h)}_{(c_h^2)}+\chi^{(c_h^2)}_{(c_h^2)}\\
&=&h-s-1+\chi^{(c_h)}_{(c_h)}-\chi^{(c_h-1,1)}_{(c_h)}\\
&=&h-s+1.
\end{eqnarray*}

In particular in either case $\chi^\alpha_{(c_1,\ldots,c_h)}\not=0$ and so
$(c_1,\ldots,c_h)$ is not $p$-vanishing.

\item
Write $(c_1,\ldots,c_h)=(c_1,\ldots,c_s,\beta_1,\ldots,\beta_r)$, with $\beta$ a
partition of the form $(1,1,1,1)$, $(2,1,1,1)$, $(2,2,1)$, $(2,2,1,1)$, $(g,1)$
or $(g,1,1)$ with $g\geq 3$. In this case let $\alpha:=(n-2,2)$. Since $n\geq 4$
it follows that $\alpha$ is a partition. Also from $p^t\nmid 2,n-1$ we have from
Lemma \ref{l5} that $p\mid\deg(\chi^\alpha)$. We will now prove that
$\chi^\alpha_{(c_1,\ldots,c_h)}\not=0$ if $(c_1,\ldots,c_h)$ ends by
$(1,1,1,1)$, $(2,1,1,1)$, $(2,2,1)$, $(2,2,1,1)$, $(g,1)$ or $(g,1,1)$ with
$g\geq 3$. As
\begin{eqnarray*}
h_{2,1}^\alpha&=&2,\\
h_{1,3}^\alpha&=&n-4
\end{eqnarray*}
(the last one only if $(1,3)\in[\alpha]$, that is $n\geq 5$), it follows from
the Murnaghan-Nakayama formula that
\[\chi^\alpha_{(c_1,\ldots,c_h)}=\chi^{(|\beta|-2,2)}_\beta.\]
For $\beta\in\{(1,1,1,1),(2,1,1,1),(2,2,1),(2,2,1,1)\}$ it is easily checked
that $\chi^{(|\beta|-2,2)}_\beta\not=0$. Also, for $g\geq 3$, we have that
\begin{eqnarray*}
\chi^{(g-1,2)}_{(g,1)}&=&-\chi^{(1)}_{(1)}=-1,\\
\chi^{(g,2)}_{(g,1,1)}&=&-\chi^{(1,1)}_{(1,1)}=-1.
\end{eqnarray*}
So $\chi^\alpha_{(c_1,\ldots,c_h)}\not=0$ and then $(c_1,\ldots,c_h)$ is not
$p$-vanishing.

\item
In this case write
$(c_1,\ldots,c_h)=(c_1,\ldots,c_s,3^w,\beta_1,\beta_2,\beta_3)$ with $c_s\geq 4$ or $(c_1,\ldots,c_h)=(3^w,\beta_1,\beta_2,\beta_3)$, where $\beta=(\beta_1,\beta_2,\beta_3)\in\{(2,1,1),(1,1,1)\}$. Since by assumption
$n\geq 6$ we have that $\alpha$ is a partition. Also since $p^t\nmid 3,n-2$
since $p^t\geq 4$ we have from Lemma \ref{l5} that $p\mid\deg(\chi^\alpha)$.

First assume that $w=0$. Then $(c_1,\ldots,c_h)=(c_1,\ldots,c_s,1,1,1)$ with
$s\geq 1$ by assumption. Since $h_{2,1}^{(n-3,3)}=3<c_s$ and $c_s+3>6$ it
follows from the Murnaghan-Nakayama formula that
\[\chi^\alpha_{(c_1,\ldots,c_h)}=\chi^{(c_s,3)}_{(c_s,1,1,1)}=-\chi^{(2,1)}_{(1,
1,1)}=-2.\]

Assume now that $w\geq 1$. Notice that $|\beta|\geq 3$. Then
\begin{eqnarray*}
\chi^\alpha_{(c_1,\ldots,c_h)}&=&\chi^{(3w+|\beta|-3,3)}_{(3^w,\beta_1,\beta_2,
\beta_3)}\\
&=&(w-1)\chi^{(3+|\beta|)}_{(3,\beta_1,\beta_2,\beta_3)}+\chi^{(|\beta|,3)}_{(3,
\beta_1,\beta_2,\beta_3)}\\
&=&\left\{\begin{array}{ll}
w-2,&\beta=(1,1,1),\\
w,&\beta=(2,1,1).
\end{array}\right.\end{eqnarray*}
Since by assumption $w\not=2$ if $\beta=(1,1,1)$ also in this case
$\chi^\alpha_{(c_1,\ldots,c_h)}\not=0$.

In particular we again have that $(c_1,\ldots,c_h)$ is not $p$-vanishing.

\item
Write now $(c_1,\ldots,c_h)=(c_1,\ldots,c_s,3^w,2,1)$ with $c_s\geq 4$ or $(c_1,\ldots,c_h)=(3^w,2,1)$.
In this case let $\alpha:=(n-4,2,2)$. Since $|\alpha|\geq g+3\geq 6$ it follows
that $\alpha$ is a partition. Also as $\{n-2,3,2\}$ is a $\beta$-set for
$\alpha$ and $p^t\geq 4$ we have from Lemma \ref{l5} that
$p\mid\deg(\chi^\alpha)$.

Assume first that $w=0$. Then $(c_1,\ldots,c_h)=(c_1,\ldots,c_s,2,1)$ with $s\geq 1$. In particular $n\geq c_s+3\geq 7$, so
that $(1,3)\in[\alpha]$ and then
\begin{eqnarray*}
h_{3,1}^\alpha&=&c_1+c_s-1\geq c_1+\ldots+c_{s-1},\\
h_{2,1}^\alpha&=&3<c_s.
\end{eqnarray*}
So
\[\chi^\alpha_{(c_1,\ldots,c_h)}=\chi^{(c_s-1,2,2)}_{(c_s,2,1)}=\chi^{(1^3)}_{(2
,1)}=-1.\]

Assume now that $w\geq 1$. As above
\[\chi^\alpha_{(c_1,\ldots,c_h)}=\chi^{(3w-1,2,2)}_{(3^w,2,1)}=-(w-1)\chi^{(5,1)
}_{(3,2,1)}+\chi^{(2^3)}_{(3,2,1)}=-1.\]

So $\chi^\alpha_{(c_1,\ldots,c_h)}\not=0$ and then $(c_1,\ldots,c_h)$ is not
$p$-vanishing.

\item
In this case write
$(c_1,\ldots,c_h)=(c_1,\ldots,c_s,4^w,\beta_1,\ldots,\beta_r)$ with $c_s\geq 5$ or $(c_1,\ldots,c_h)=(4^w,\beta_1,\ldots,\beta_r)$, where $\beta\in\{(3,3,1,1,1),(2,1,1)\}$. Also let $\alpha:=(n-4,4)$.
By assumption $n\geq 8$, so that $\alpha$ is a partition. Also since $p^t\geq
5$, so that $p^t\nmid 4,n-3$, we have from Lemma \ref{l5} that
$p\mid\deg(\chi^\alpha)$.

Assume first that $\beta=(3,3,1,1,1)$. Since $\sum_{c_i\leq 4}c_i\geq 9$ and
$h_{2,1}^\alpha=4$, it follows from the Murnaghan-Nakayama formula that
\[\chi^\alpha_{(c_1,\ldots,c_h)}=\chi^{(4w+5,4)}_{(4^w,3,3,1,1,1)}=w\chi^{(9)}_{
(3,3,1,1,1)}+\chi^{(5,4)}_{(3,3,1,1,1)}=w+3\not=0.\]

Assume next that $\beta=(2,1,1)$ and that $w\geq 1$. Then $w\geq 2$ by
assumption. Also in this case $\sum_{c_i\leq 4}c_i\geq 8$ and so, similarly to
before,
\[\chi^\alpha_{(c_1,\ldots,c_h)}=\chi^{(4w,4)}_{(4^w,2,1,1)}=(w-1)\chi^{(8)}_{(4
,2,1,1)}+\chi^{(4,4)}_{(4,2,1,1)}=w-1\not=0.\]

At last assume that $\beta=(2,1,1)$ and $w=0$. Then $(c_1,\ldots,c_h)=(c_1,\ldots,c_s,\beta_1,\ldots,\beta_r)$ with $s\geq 1$. As $\sum_{i\geq
s}c_i\geq 9$ and $c_s>4$ we have that
\[\chi^\alpha_{(c_1,\ldots,c_h)}=\chi^{(c_s,4)}_{(c_s,2,1,1)}=-\chi^{(3,1)}_{(2,
1,1)}=-1.\]

It follows that $(c_1,\ldots,c_h)$ is not $p$-vanishing in each of the above
cases.

\item
Write now $(c_1,\ldots,c_h)=(c_1,\ldots,c_s,4,2,1,1)$ with $c_s\geq 5$ or $(c_1,\ldots,c_h)=(4,2,1,1)$.

If $(c_1,\ldots,c_h)=(4,2,1,1)$ then $n=8$ and so $p^t=8$, since $5\leq p^t\mid n$. Since
$\deg(\chi^{(3,3,2)})=42$ and $\chi^{(3,3,2)}_{(4,2,1,1)}=-2$ it follows that
$(4,2,1,1)$ is not 2-vanishing.

So assume now that $(c_1,\ldots,c_h)=(c_1,\ldots,c_s,4,2,1,1)$ with $s\geq 1$. In this case let
$\alpha:=(n-5,3,2)$. By assumption $n>8$, so that $\alpha$ is a partition. Also
since $p^t\geq 5$, so that $p^t\nmid 2,4,n-3$, we have from Lemma \ref{l5} that
$p\mid\deg(\chi^\alpha)$.

Since $n>8$, so that $(1,4)\in[\alpha]$ and
\begin{eqnarray*}
h_{1,4}^\alpha&=&n-8=c_1+\ldots+c_s,\\
h_{2,1}^\alpha&=&4<c_s
\end{eqnarray*}
we have that
\[\chi^\alpha_{(c_1,\ldots,c_h)}=\chi^{(3,3,2)}_{(4,2,1,1)}=-2\]
and then also in this case $(c_1,\ldots,c_h)$ is not $p$-vanishing.

\item
The case $(c_1,\ldots,c_h)=(3,3,1,1,1)$ cannot happen, since $p^t=4$. Write
$(c_1,\ldots,c_h)=(c_1\ldots,c_s,6^w,5^x,4^y,3,3,1,1,1)$ with $s=0$ or $c_s\geq
7$. By assumption $s+w+x+y\geq 1$. Let $\alpha:=(n-6,3,1,1,1)$. From $n>9$ it
follows that $\alpha$ is a partition. Using the hook formula we have that
\[\deg(\chi^\alpha)=\frac{\prod_{1\leq h\leq
n:h\not\in\{h_{1,j}^\alpha\}}h}{\prod_{(i,j)\in[\alpha]:i\geq
2}h_{i,j}^\alpha}=\frac{n(n-1)(n-3)(n-4)(n-5)(n-8)}{6\cdot 3\cdot 2\cdot 2\cdot
1\cdot 1}.\]
From $4=p^t\mid n$ it follows that $2\mid\deg(\chi^\alpha)$.

Since $n>9$ we also have that $(1,4)\in[\alpha]$. So
\begin{eqnarray*}
h_{1,4}^\alpha&=&c_1+\ldots+c_s+6w+5x+4y,\\
h_{2,1}^\alpha&=&6,\\
h_{2,2}^\alpha&=&2,\\
h_{3,1}^\alpha&=&3
\end{eqnarray*}
and then
\[\chi^\alpha_{(c_1,\ldots,c_h)}=\chi^{(3+6w,3,1,1,1)}_{(6^w,3,3,1,1,1)}=-w\chi^
{(9)}_{(3,3,1,1,1)}+\chi^{(3,3,1,1,1)}_{(3,3,1,1,1)}=-w-3\not=0.\]
In particular $(c_1,\ldots,c_h)$ is not 2-vanishing.
\end{enumerate}
\end{proof}

\section{Proof of Theorem \ref{t19}}\label{s3}

The next theorem is stronger than Theorem \ref{t19} and so proving it will also
prove Theorem \ref{t19}.

\begin{theor}
Let $p=2$ or $p=3$ and $t\geq 0$. If $(c_1,\ldots,c_h)\vdash n$ is $p$-vanishing, then we
have that $\sum_{c_j\geq p^t}c_i\geq d_tp^t$ if one of the following holds:
\begin{itemize}
\item
$p=2$ and $n$ is odd or $8\mid n$.

\item
$p=2$, $n\equiv 2\mbox{ mod }4$ and $t\not=1$.

\item
$p=2$, $n\equiv 4\mbox{ mod }8$ and $t\not=1,2$.

\item
$p=3$ and $n\equiv 0,1,2,4\mbox{ or }7\mbox{ mod }9$.

\item
$p=3$, $n\equiv 3,5,6\mbox{ or }8\mbox{ mod }9$ and $t\not=1$.
\end{itemize}
\end{theor}

\begin{proof}
Notice that the theorem clearly holds for $t=0$, as in this case $p^t=1$. So we
will assume that $t\geq 1$. The theorem also clearly holds if $d_t=0$. So we
will assume that $d_t\not=0$. Let $(c_1,\ldots,c_h)$ be $p$-vanishing. We will
prove that then $\sum_{c_j\geq p^t}c_i\geq d_tp^t$ if we are in one of the cases
above. Using Lemma \ref{l23} it will be enough to prove that
$\sum_{c_j<e_t}c_j\leq e_t$.

The proof of the theorem will be divided in the following cases:

\begin{enumerate}
\item\label{i1}
$t\geq 2$ if $p=3$ and $n\equiv 2\mbox{ mod }3$ or $t\geq 1$ otherwise,
$p^t\nmid n$ and the theorem holds for $t-1$.

\item
$8\mid n$ if $p=2$ or $9\mid n$ if $n=3$ and $p^t\mid n$.

\item
$p=2$, $n\equiv 4\mbox{ mod }8$ and $t=3$.

\item
$p=2$, $n\equiv 2\mbox{ mod }4$ and $t=2$ or $p=3$, $n\equiv 2\mbox{ mod }9$ and
 $t=1$.

\item
$p=3$, $n\equiv 5\mbox{ or }8\mbox{ mod }9$ and $t=2$.

\item
$p=3$, $n\equiv 3\mbox{ or }6\mbox{ mod }9$ and $t=2$.

\end{enumerate}

We will now prove the result in each of the above cases.

\begin{enumerate}
\item
We can write $n=d_tp^t+a_{t-1}p^{t-1}+e_{t-1}$. By assumption on $t-1$ and by
Theorem \ref{t17} we have that
\[\sum_{i:c_i\geq p^{t-1}}c_i=d_tp^t+a_{t-1}p^{t-1}\]
and so, from Theorem \ref{t15}, that, for some $l\geq 0$,
\begin{equation}\label{eq5}
(c_1,\ldots,c_h)\vdash (p^{t-1}f_1,\ldots,p^{t-1}f_l,c_{l+1},\ldots,c_h)
\end{equation}
with $(f_1,\ldots,f_l)\vdash d_tp+a_{t-1}$ and $c_{l+1}<p^{t-1}$. Notice that
$(c_{l+1},\ldots,c_h)$ is a partition of $e_{t-1}$. Proving that
$\sum_{c_j<e_t}c_j\leq e_t$ is then equivalent to proving that
\[\sum_{p^{t-1}\leq c_j<e_t}c_j\leq e_t-e_{t-1}=a_{t-1}p^{t-1},\]
which in turn is equivalent to
\[\sum_{f_j<e_t/p^{t-1}}f_j\leq a_{t-1}\]
from Equation \eqref{eq5}. Since
\[e_t/p^{t-1}=a_{t-1}+e_{t-1}/p^{t-1}<a_{t-1}+1\]
it is enough to prove that
\[\sum_{f_j\leq a_{t-1}}f_j\leq a_{t-1}.\]
This clearly holds if $a_{t-1}=0$. So we can assume that $a_{t-1}>0$. Let
$\alpha$ be the partition with
\begin{eqnarray*}
\alpha_{(p^{t-1})}&=&(e_{t-1}),\\
\alpha^{(p^{t-1})}&=&((d_tp,1^{a_{t-1}}),(0),\ldots,(0)).
\end{eqnarray*}
Then $\alpha\vdash n$. From Lemma \ref{l22} and Corollary \ref{l16} applied to
both $\alpha$ and $(d_tp,1^{a_{t-1}})$ we have that $p\mid\deg(\chi^\alpha)$. By
assumption and from Lemma \ref{l3} we have that
\[0=\chi^\alpha_{(c_1,\ldots,c_h)}=\pm\chi^{\alpha^{(p^{t-1})}}_{(f_1,\ldots,
f_l)}\chi^{\alpha_{(p^{t-1})}}_{(c_{l+1},\ldots,c_h)}=\pm\chi^{(d_tp,1^{a_{t-1}}
)}_{(f_1,\ldots,f_l)}.\]
Notice that by assumption $1\leq a_{t-1}\leq 2$, since $1\leq a_{t-1}<p\leq 3$.

Assume first that $a_{t-1}=1$. Then
\[\chi^{(d_tp,1)}_{(f_1,\ldots,f_l)}=0\]
and so
\[\sum_{f_j\leq 1}f_j=\sum_{f_j=1}f_j=\chi^{(d_tp,1)}_{(f_1,\ldots,f_l)}+1=1,\]
since $\chi^{(x-1,1)}_{(x^{w_x},\ldots,1^{w_1})}=w_1-1$ for $x\geq 2$ (see
2.3.16 of \cite{b1}).

Assume now that $a_{t-1}=2$ and that $\sum_{f_j\leq 2}f_j>2$. We will show that
this give a contradiction. Write
\[(f_1,\ldots,f_l)=(f_1,\ldots,f_m,2^g,1^h)\]
with $m=0$ or $f_m\geq 3$. Then, since by assumption $2g+h\geq 3$,
\begin{eqnarray*}
0&=&\chi^{(d_tp,1^2)}_{(f_1,\ldots,f_l)}\\
&=&\chi^{(2g+h-2,1^2)}_{(2^g,1^h)}\\
&=&-g\chi^{(h)}_{(1^h)}+\delta_{h\geq 3}\chi^{(h-2,1^2)}_{(1^h)}\\
&=&-g+\delta_{h\geq 3}\frac{(h-1)(h-2)}{2}.
\end{eqnarray*}
In particular $h\geq 3$ since $2g+h\geq 3$. Let $\beta$ be given by
\begin{eqnarray*}
\beta_{(p^{t-1})}&=&(e_{t-1}),\\
\beta^{(p^{t-1})}&=&((d_tp,1),(1),(0),\ldots,(0)).
\end{eqnarray*}
Then $\beta\vdash n$ (in this case $p=3$ and $t-1\geq 2$, so that $p^{t-1}>1$).
Since we cannot remove $d_t$ hooks of length $p$ from $\beta^{(p^{t-1})}$ (there
are only $d_t-1$ nodes of $\beta^{(p^{t-1})}$ with hook length divisible by $p$)
we have from Lemmas \ref{l22} and \ref{l1} that $p\mid\deg(\chi^\beta)$. From
Lemma \ref{l3} we then have that
\[0=\chi^\alpha_{(c_1,\ldots,c_h)}=\pm\chi^{\alpha^{(p^{t-1})}}_{(f_1,\ldots,
f_l)}\chi^{\alpha_{(p^{t-1})}}_{(c_{l+1},\ldots,c_h)}=\pm\chi^{((d_tp,1),(1))}_{
(f_1,\ldots,f_l)}.\]
Since $h\geq 3$ we have that
\[0=\chi^{((d_tp,1),(1))}_{(f_1,\ldots,f_m,2^g,1^h)}=\chi^{((h-2,1),(1))}_{(1^h)
}\]
which gives a contradiction. In particular $2g+h\leq 2$ and so the result is
proved in this case too.

\item
In these cases the theorem follows from Theorem \ref{t13}.

\item
This case follows from Theorem \ref{t13}.

\item
We will show that in this case $\sum_{c_j=1}c_j\leq 2$. By assumption $d_t>0$,
so that $n>2$. As $n\equiv 2\mbox{ mod }4$ or $n\equiv 2\mbox{ mod }9$ we then
have that $n\geq 6$. In particular $(n-2,1,1)$, $(n-3,2,1)$ and $(n-3,1,1,1)$
are partitions. The degrees of $\chi^{(n-2,1,1)}$ and $\chi^{(n-3,1,1,1)}$ are
divisible by $p$ from Corollary \ref{l16}. Using the hook formula it can be
easily seen that the degree of $\chi^{(n-3,2,1)}$ is $n(n-2)(n-4)/3$ and then it
is divisible by $p$. Write $(c_1,\ldots,c_h)=(c_1,\ldots,c_l,3^c,2^b,1^a)$ with
$l=0$ or $c_l\geq 4$ and assume that $x\geq 3$. Then, as $(c_1,\ldots,c_h)$ is
$p$-vanishing we have that
\[0=\chi^{(n-2,1,1)}_{(c_1,\ldots,c_h)}=\chi^{(a+2b-2,1,1)}_{(2^b,1^a)}=-b+\chi^
{(a-2,1,1)}_{(1^a)}=-b+\frac{(a-1)(a-2)}{2}.\]
So $b=(a-1)(a-2)/2\geq 1$ and then $2b+a\geq 5$. In particular
\begin{eqnarray*}
\chi^{(n-3,1,1,1)}_{(c_1,\ldots,c_h)}&=&\chi^{(3c+2b+a-3,1,1,1)}_{(3^c,2^b,1^a)}
\\
&=&c\chi^{(a)}_{(1^a)}-b\chi^{(a-1,1)}_{(1^a)}+\delta_{a\geq
4}\chi^{(a-3,1,1,1)}_{(1^a)}\\
&=&c-b(a-1)+\frac{(a-1)(a-2)(a-3)}{6}
\end{eqnarray*}
and
\begin{eqnarray*}
\chi^{(n-3,2,1)}_{(c_1,\ldots,c_h)}&=&\chi^{(3c+2b+a-3,2,1)}_{(3^c,2^b,1^a)}\\
&=&-c\chi^{(2b+3a)}_{(2^b,1^a)}+\chi^{(2b+a-3,2,1)}_{(2^b,1^a)}\\
&=&-c+\chi^{(2b+a-3,2,1)}_{(2^b,1^a)},
\end{eqnarray*}
as
\begin{eqnarray*}
h_{1,2}^{(n-3,1,1,1)}&=&n-4\geq c_1+\ldots+c_l+3c,\\
h_{1,3}^{(n-3,2,1)}&=&n-5\geq c_1+\ldots+c_l+3c.
\end{eqnarray*}
As $a\geq 3$ and as $b=1$ if $a=3$ and $b=3$ if $a=4$, we have that
\[\chi^{(2b+a-3,2,1)}_{(2^b,1^a)}=\left\{\begin{array}{ll}
-1&\mbox{if }a=3\\
0&\mbox{if }a=4\\
\chi^{(a-3,2,1)}_{(1^a)}&\mbox{if }a\geq 5
\end{array}\right.\]
and then
\[\chi^{(2b+a-3,2,1)}_{(2^b,1^a)}=\frac{a(a-2)(a-4)}{3}.\]
As $(c_1,\ldots,c_h)$ is $p$-vanishing and $p$ divides the degrees of
$\chi^{(n-3,1,1,1)}$ and $\chi^{(n-3,2,1)}$, we have that $c=a(a-2)(a-4)/3$ and
that
\begin{eqnarray*}
0&=&c-b(a-1)+\frac{(a-1)(a-2)(a-3)}{6}\\
&=&\frac{a(a-2)(a-4)}{3}-\frac{(a-1)^2(a-2)}{2}+\frac{(a-1)(a-2)(a-3)}{6}\\
&=&-a(a-2)
\end{eqnarray*}
which gives a contradiction since we assumed that $a\geq 3$. So
$\sum_{c_j=1}c_j\leq 2$.

\item
Assume now that $p=3$, $n\equiv 5\mbox{ or }8\mbox{ mod }9$ and $t=2$. By
assumption on $d_t$ it follows that $n\geq 14$. In particular $(n-2,1,1)$,
$(n-4,2,1,1)$, $(n-5,3,2)$, $(n-4,2,2)$, $(n-6,3,3)$, $(n-5,1^5)$,
$(n-7,2^2,1^3)$ and $(n-7,2,1^5)$ are partitions.

If the result does not hold for $t=2$, then we have from Case \ref{i1} that the
result does not hold for $t=1$ either. In particular $\sum_{c_j<2}c_j>2$ and
then from Lemma \ref{l23} we have that $\sum_{c_j=1}c_j\geq 3$. So we can write
$(c_1,\ldots,c_h)=(c_1,\ldots,c_l,4^d,3^c,2^b,1^a)$, with $l=0$ or $c_l\geq 5$
and with $a\geq 3$. As in the previous case we have from Corollary \ref{l16}
that 3 divides the degree of $\chi^{(n-2,1,1)}$ and that
\begin{equation}\label{eq6}
0=\chi^{(n-2,1,1)}_{(c_1,\ldots,c_h)}=\chi^{(a+2b-2,1,1)}_{(2^b,1^a)}=-b+\frac{
(a-1)(a-2)}{2},
\end{equation}
so that $b=(a-1)(a-2)/2$.

From the hook formula we have that the degree of $\chi^{(n-4,2,1,1)}$ is
\[n(n-2)(n-3)(n-5)/8\]
and so it is divisible by 3.  For $a\geq 4$, we have that $b\geq 3$ and then
\[a+2b\geq 10>6=\sum_{i:(n-4,2,1,1)'_i>1}(n-4,2,1,1)'_i,\]
and as $(c_1,\ldots,c_h)$ is 3-vanishing we have that
\begin{eqnarray*}
0&=&\chi^{(n-4,2,1,1)}_{(c_1,\ldots,c_h)}\\
&=&\chi^{(a+2b+4d-4,2,1,1)}_{(4^d,2^b,1^a)}\\
&=&d\chi^{(a+2b)}_{(2^b,1^a)}+\chi^{(a+2b-4,2,1,1)}_{(2^b,1^a)}\\
&=&d+\chi^{(a+2b-4,2,1,1)}_{(2^b,1^a)}.
\end{eqnarray*}
For $a\geq 6$ we have that
\begin{eqnarray*}
\chi^{(a+2b-4,2,1,1)}_{(2^b,1^a)}&=&-\binom{b}{2}\chi^{(a)}_{(1^a)}-b\chi^{(a-2,
2)}_{(1^a)}+\chi^{(a-4,2,1,1)}_{(1^a)}\\
&=&-\frac{b(b-1)}{2}-\frac{ab(a-3)}{2}+\frac{a(a-2)(a-3)(a-5)}{8}\\
&=&-\frac{a^4-4a^3+a^2+6a}{4}
\end{eqnarray*}
since $b=(a-1)(a-2)/2$. For $4\leq a\leq 5$ we have that
\[\chi^{(a+2b-4,2,1,1)}_{(2^b,1^a)}=\left\{\begin{array}{ll}
-10&\mbox{if }a=4,\\
-45&\mbox{if }a=5.
\end{array}\right.\]
Assume now that $a=3$, so that $b=1$. If $d\geq 1$ then
\[\chi^{(n-4,2,1,1)}_{(c_1,\ldots,c_h)}=\chi^{(4d+1,2,1,1)}_{(4^d,2,1^3)}
=(d-1)\chi^{(9)}_{(4,2,1^3)}+\chi^{(5,2,1,1)}_{(4,2,1^3)}=d\not=0\]
which brings a contradiction with $\chi^{(n-4,2,1,1)}_{(c_1,\ldots,c_h)}=0$ by
assumption.

So
\[d=\left\{\begin{array}{ll}
0&\mbox{if } a=3\\
10&\mbox{if } a=4\\
45&\mbox{if } a=5\\
\frac{a^4-4a^3+a^2+6a}{4}&\mbox{if } a\geq 6.
\end{array}\right.\]

Let's now consider $(n-5,3,2)$. We have that the degree of $\chi^{(n-5,3,2)}$ is
$n(n-1)(n-2)(n-5)(n-7)/24$, which is divisible by 3. If $a\geq 4$ then $a+2b\geq
10$ and so, using Equation \eqref{eq6}, we have that
\begin{eqnarray*}
0&=&\chi^{(n-5,3,2)}_{(c_1,\ldots,c_h)}\\
&=&\chi^{(a+2b+3c+4d-5,3,2)}_{(4^d,3^c,2^b,1^a)}\\
&=&-d\chi^{(a+2b-1,1)}_{(2^b,1^a)}-c\chi^{(a+2b-2,1,1)}_{(2^b,1^a)}+\chi^{
(a+2b-5,3,2)}_{(2^b,1^a)}\\
&=&-d(a-1)+\chi^{(a+2b-5,3,2)}_{(2^b,1^a)}.
\end{eqnarray*}
Since $b=(a-1)(a-2)/2$ for $a\geq 8$ we have that
\begin{eqnarray*}
\chi^{(a+2b-5,3,2)}_{(2^b,1^a)}&=&\binom{b}{2}\chi^{(a-1,1)}_{(1^a)}+b\chi^{(a-3
,3)}_{(1^a)}+\chi^{(a-5,3,2)}_{(1^a)}\\
&=&\frac{(a-1)b(b-1)}{2}+\frac{a(a-1)(a-5)b}{6}\\
&&\hspace{12pt}+\frac{a(a-1)(a-2)(a-5)(a-7)}{24}\\
&=&\frac{a^5-9a^4+29a^3-39a^2+18a}{4}.
\end{eqnarray*}
For $4\leq a\leq 7$ we have that
\[\chi^{(a+2b-5,3,2)}_{(2^b,1^a)}=\left\{\begin{array}{ll}
6&\mbox{if }a=4,\\
60&\mbox{if }a=5,\\
270&\mbox{if }a=6,\\
840&\mbox{if }a=7.
\end{array}\right.\]
So, using the previous formulas for $d$, 
\[0=-d(a-1)+\chi^{(a+2b-5,3,2)}_{(2^b,1^a)}=\left\{\begin{array}{ll}
-24&\mbox{if }a=4,\\
-120&\mbox{if }a=5,\\
-360&\mbox{if }a=6,\\
-840&\mbox{if }a=7,\\
-a(a-1)(a-2)(a-3)&\mbox{if }a\geq 8
\end{array}\right.\]
which gives a contradiction. So $a=3$ and then $b=1$ and $d=0$, that is
$(c_1,\ldots,c_h)=(c_1,\ldots,c_l,3^c,2,1^3)$, with $c_l\geq 5$.

Assume first that $n\equiv 5\mbox{ mod }9$. Then
\[3\,\left|\, \frac{n(n-1)(n-4)(n-5)}{12}\right.=\deg(\chi^{(n-4,2,2)}).\]
If $c\geq 1$ then
\[\chi^{(n-4,2,2)}_{(c_1,\ldots,c_h)}=\chi^{(3c+1,2,2)}_{(3^c,2,1^3)}
=-(c-1)\chi^{(7,1)}_{(3,2,1^3)}+\chi^{(4,2,2)}_{(3,2,1^3)}=-2c\not=0\]
which contradicts $(c_1,\ldots,c_h)$ being 3-vanishing by assumption. So $c=0$
and then $\sum_{c_j<5}c_j=5$ in this case and then the result holds.

Assume now that $n\equiv 8\mbox{ mod }9$. Then
\[3\,\left|\,\frac{n(n-1)(n-2)(n-3)(n-7)(n-8)}{144}\right.=\deg(\chi^{(n-6,3,3)}
).\]
If $c=0$ then
\[\chi^{(n-6,3,3)}_{(c_1,\ldots,c_h)}=\chi^{(c_l-1,3,3)}_{(c_l,2,1^3)}=\chi^{(2,
2,1)}_{(2,1^3)}=-1\]
which contradicts $(c_1,\ldots,c_h)$ being 3-vanishing. If $c\geq 2$ then
\begin{eqnarray*}
\chi^{(n-6,3,3)}_{(c_1,\ldots,c_h)}\!\!&=&\!\chi^{(3c-1,3,3)}_{(3^c,2,1^3)}\\
&=&\!2\binom{c-2}{2}\chi^{(11)}_{(3^2,2,1^3)}+(c-2)\left(\chi^{(8,3)}_{(3^2,2,
1^3)}-\chi^{(8,2,1)}_{(3^2,2,1^3)}\right)+\chi^{(5,3,3)}_{(3^2,2,1^3)}\\
&=&\!(c-2)(c+2)+3\\
&\not=&\!0
\end{eqnarray*}
which also contradicts $(c_1,\ldots,c_h)$ being 3-vanishing. So $c=1$ and we can
write $(c_1,\ldots,c_h)=(c_1,\ldots,c_{l'},7^g,6^f,5^e,3,2,1^3)$ with $l'=0$ or
$c_{l'}\geq 8$.

Consider now $(n-5,1^5)$. The degree of the corresponding character is
\[\frac{(n-1)(n-2)(n-3)(n-4)(n-5)}{120}\]
and so it is divisible by 3. Since $(c_1,\ldots,c_h)$ is 3-vanishing we have
that
\[0=\chi^{(n-5,1^5)}_{(c_1,\ldots,c_h)}=\chi^{(5e+3,1^5)}_{(5^e,3,2,1^3)}=e\chi^
{(8)}_{(3,2,1^3)}+\chi^{(3,1^5)}_{(3,2,1^3)}=e\]
and so $e=0$.

Next consider $(n-7,2^2,1^3)$. The corresponding character has degree
\[\frac{n(n-1)(n-3)(n-4)(n-5)(n-7)(n-8)}{360}\]
which is divisible by 3, since $n\equiv 8\mbox{ mod }9$. If $f\geq 1$ then
\[\chi^{(n-7,2^2,1^3)}_{(c_1,\ldots,c_h)}=\chi^{(6f+1,2^2,1^3)}_{(6^f,3,2,1^3)}
=(f-1)\chi^{(13,1)}_{(6,3,2,1^3)}+\chi^{(7,2^2,1^3)}_{(6,3,2,1^3)}=2f\not=0\]
which contradicts $(c_1,\ldots,c_h)$ being 3-vanishing. So $f=0$.

At last consider $(n-7,2,1^5)$. The corresponding character has degree
\[\frac{n(n-2)(n-3)(n-4)(n-5)(n-6)(n-8)}{840}\]
which is divisible by 3. If $g\geq 1$ then
\[\chi^{(n-7,2,1^5)}_{(c_1,\ldots,c_h)}=\chi^{(7g+1,2,1^5)}_{(7^g,3,2,1^3)}
=-(g-1)\chi^{(15)}_{(7,3,2,1^3)}+\chi^{(8,2,1^5)}_{(7,3,2,1^3)}=-g\not=0\]
which gives to a contradiction. So
$(c_1,\ldots,c_h)=(c_1,\ldots,c_{l'},3,2,1^3)$ with $c_{l'}\geq 8$ (as $n>6$ we
have $l'>0$) and then the result follows from Lemma \ref{l23}.

\item
Assume now that $p=3$, $n\equiv 3\mbox{ or }6\mbox{ mod }9$ and $t=2$. From
$d_t>0$ we have that $n\geq 12$, so that $(n-2,2)$, $(n-4,2,1,1)$, $(n-5,3,1,1)$
and $(n-6,2,2,2)$ are partitions of $n$. If $n\equiv 3\mbox{ mod }9$ we will
show that $\sum_{c_j<3}c_j\leq 3$, while if $n\equiv 6\mbox{ mod }9$ that
$\sum_{c_j<6}c_j\leq 6$.

Write $(c_1,\ldots,c_h)=(c_1,\ldots,c_l,2^b,1^a)$ with $l=0$ or $c_l\geq 3$.
From Case \ref{i1}, if the result does not hold for $t=2$, then it can not hold
for $t=1$ either and then $\sum_{c_j<3}c_j>0$. From Theorem \ref{t13} we then
have that $a\geq 1$.

From the hook formula we have that the degree of $\chi^{(n-2,2)}$ is $n(n-3)/2$
and so by assumption it is divisible by 3. If $a+2b\leq 3$, then $l\geq 1$ and
\[\chi^{(n-2,2)}_{(c_1,\ldots,c_h)}=\chi^{(c_l+a+2b-2,2)}_{(c_l,2^b,1^a)}=\left\{\begin{array}{ll}
-\chi^{(1)}_{(1)}=-1&\mbox{if }a=1,\,\,b=0,\\
0&\mbox{if }a=1,\,\,b=1,\\
-\chi^{(1^2)}_{(1^2)}=-1&\mbox{if }a=2,\,\,b=0,\\
0&\mbox{if }a=3,\,\,b=0.
\end{array}\right.\]
If instead $a+2b\geq 4$, then we have that
\[\chi^{(n-2,2)}_{(c_1,\ldots,c_h)}=\chi^{(a+2b-2,2)}_{(2^b,1^a)}=\left\{\begin{array}{ll}
(b-2)\chi^{(5)}_{(2^2,1)}+\chi^{(3,2)}_{(2^2,1)}=b-1&\mbox{if }a=1,\\
(b-1)\chi^{(4)}_{(2,1^2)}+\chi^{(2,2)}_{(2,1^2)}=b-1&\mbox{if }a=2,\\
(b-1)\chi^{(5)}_{(2,1^3)}+\chi^{(3,2)}_{(2,1^3)}=b&\mbox{if }a=3,\\
b\chi^{(a)}_{(1^a)}+\chi^{(a-2,2)}_{(1^a)}=b+a(a-3)/2&\mbox{if }a\geq
4.
\end{array}\right.\]
As $(c_1,\ldots,c_h)$ is 3-vanishing, so that
$\chi^{(n-2,2)}_{(c_1,\ldots,c_h)}=0$, we then have that
$(a,b)\in\{(1,1),(2,1),(3,0)\}$. It also follows that $l>0$, as $n>4\geq a+2b$.

Consider next $(n-4,2,1,1)$. From the hook formula we have that
$\deg(\chi^{(n-4,2,1,1)})=n(n-2)(n-3)(n-5)/8$ and so it is divisible by 3. Write
now $(c_1,\ldots,c_h)=(c_1,\ldots,c_{l'},4^d,3^c,2^b,1^a)$, with $c_{l'}\geq 5$.
If $d\geq 1$ then
\begin{eqnarray*}
\chi^{(n-4,2,1,1)}_{(c_1,\ldots,c_h)}&=&\chi^{(a+2b+4(d-1),2,1,1)}_{(4^d,2^b,
1^a)}\\
&=&(d-1)\chi^{(a+2b+4)}_{(4,2^b,1^a)}+\chi^{(a+2b,2,1,1)}_{(4,2^b,1^a)}\\
&=&\left\{\begin{array}{ll}
d&\mbox{if }a+2b=3,\\
d+1&\mbox{if }a+2b=4.
\end{array}\right.\\
&\not=&0
\end{eqnarray*}
which gives a contradiction with $(c_1,\ldots,c_h)$ being 3-vanishing. So $d=0$.
Assume now that $a+2b=4$ (and so $(a,b)=(2,1)$). Then
\[\chi^{(n-4,2,1,1)}_{(c_1,\ldots,c_h)}=\chi^{(c_l,2,1,1)}_{(c_l,2,1,1)}=-\chi^{
(1^4)}_{(2,1,1)}=1\not=0\]
which also gives a contradiction. So
\[(c_1,\ldots,c_h)=(c_1,\ldots,c_{l'},3^c,2^b,1^a)\]
with $a+2b=3$ and $l'=0$ or $c_{l'}\geq 5$. If $n\equiv 3\mbox{ mod }9$ we are
done, due to Theorem \ref{l23}. So assume now that $n\equiv 6\mbox{ mod }9$.

From the hooks formula we have that
\[\deg(\chi^{(n-5,3,1,1)})=\frac{n(n-1)(n-3)(n-4)(n-7)}{20},\]
which is divisible by 3. Write
$(c_1,\ldots,c_h)=(c_1,\ldots,c_{l''},5^e,3^c,2^b,1^a)$ with $l''=0$ or
$c_{l''}\geq 6$. Since $a+2b=3$ if $e\geq 1$, then
\[\chi^{(n-5,3,1,1)}_{(c_1,\ldots,c_l)}=\chi^{(5e-2,3,1,1)}_{(5^e,2^b,1^a)}
=(e-1)\chi^{(8)}_{(5,2^b,1^a)}+\chi^{(3,3,1,1)}_{(5,2^b,1^a)}=e\not=0\]
which gives a contradiction. So
$(c_1,\ldots,c_h)=(c_1,\ldots,c_{l''},3^c,2^b,1^a)$.

Consider now the partition $(n-6,2,2,2)$. We have that
\[\deg(\chi^{(n-6,2,2,2)})=\frac{n(n-1)(n-2)(n-5)(n-6)(n-7)}{144}\]
which is divisible by 3 since $n\equiv 6\mbox{ mod }9$. Assume that $c\geq 2$.
Then
\begin{eqnarray*}
\chi^{(n-6,2,2,2)}_{(c_1,\ldots,c_h)}\!\!\!&=&\!\!\!\chi^{(3(c-1),2,2,2)}_{(3^c,
2^b,1^a)}\\
&=&\!\!\!2\binom{c\!-\!2}{2}\chi^{(9)}_{(3^2,2^b,1^a)}
\!+\!(c\!-\!2)\!\left(\chi^{(6,1,1,1)}_{(3^2,2^b,1^a)}\!-\!\chi^{(6,2,1)}_{(3^2,
2^b,1^a)}\right)\!+\!\chi^{(3,2,2,2)}_{(3,2^b,1^a)}\\
&=&\!\!\!\left\{\begin{array}{ll}
\!(c\!-\!2)(c\!+\!2)\!+\!3&\mbox{if }a=3,\,\,b=0,\\
\!(c\!-\!2)c\!+\!1&\mbox{if }a=1,\,\,b=1
\end{array}\right.\\
&\not=&\!\!\!0
\end{eqnarray*}
which gives a contradiction with $(c_1,\ldots,c_h)$ being 3-vanishing. In
particular $c\leq 1$, so $\sum_{c_j<6}c_j\leq 6$ and the result holds.
\end{enumerate}
\end{proof}

\section*{Acknowledgements}

The work contained in this paper is part of the author's master thesis
(\cite{m1}), which was written at the University of Copenhagen, under the supervision of J\o rn B. Olsson, whom the author would like to thank for his help in reviewing the paper.

While writing the paper the author was supported by the DFG grant for the Graduiertenkolleg Experimentelle und konstruktive Algebra at RWTH Aachen University (GRK 1632).

\end{document}